\documentclass{lmcs}
\pdfoutput=1

\usepackage[utf8]{inputenc}
\usepackage{lastpage}
\lmcsdoi{17}{3}{4}
\lmcsheading{}{\pageref{LastPage}}{}{}%
{Nov.~21,~2019}{Jul.~20,~2021}{}

\usepackage{mathtools, amssymb, amsfonts, amsthm, enumitem, mathdots, enumitem, pdfsync}
\usepackage{epsfig, psfrag, color, multicol, graphicx, tikz, bm}
\usepackage{amsfonts}

\usepackage{bookmark}

\theoremstyle{plain}
\theoremstyle{thmC}
\newtheorem{factC}[thm]{Fact}

\theoremstyle{definition}

\numberwithin{equation}{section}

\def\sq{\square}

\def\zz{\mathbb Z}

\def\rr{\mathbb R}
\def\qqq{\mathbb Q}

\def\ga{\gamma}

\def\al{\alpha}
\def\be{\beta}

\def\T{\mathbf{T}}

\def\wt{\widetilde}
\def\<{\langle}
\def\>{\rangle}

\def\Z{ {\text {\rm Z} } }

\def\Q{{\text {\rm Q} } }

\def\0{{\mathbf 0}}

\def\NN{{\mathbb N}}

\def\.{\hskip.06cm}
\def\ts{\hskip.03cm}

\def\poly{\textup{\textsf{P}}}
\def\pspace{\textup{\textsf{PSPACE}}}
\def\PSPACE{\textup{\textsf{PSPACE}}}
\def\SigmaP{\Sigma^{\poly}}

\def\Z{\mathbb{Z}}
\def\R{\mathbb{R}}
\def\N{\mathbb{N}}
\def\Q{\mathbb{Q}}

\def\T{\bm T}

\newcommand{\cj}[1]{\overline{#1}}

\renewcommand{\b}{\cj{b}}

\def\a{\cj{a}}
\def\c{\cj{c}}
\def\d{\cj{d}}

\newcommand{\x}{\mathbf{x}}
\renewcommand{\t}{\mathbf{t}}

\def\w{\mathbf{w}}

\newcommand{\Qrc}{\mathbb{Q}_\textsf{\textup{alg}}}

\newcommand{\y}{\mathbf{y}}
\newcommand{\z}{\mathbf{z}}

\newcommand{\polyin}{\textup{poly}}

\newcommand{\ex}{\exists\ts}
\renewcommand{\for}{\forall\ts}

\def\nin{\noindent}

\renewcommand\L{\mathcal{L}}

\def\NP{{\textup{\textsf{NP}}}}

\def\polyH{\textup{\textsf{PH}}}

\def\s{\mathbf{s}}

\renewcommand{\mod}[1]{
\;\, (\textup{mod} \; #1)
}

\def\cal{\mathcal}

\def\phi{\varphi}

\def\markerblock{[\xsymb, \xsymb]}
\def\when{\quad \text{if} \quad}
\def\and{\quad \text{and} \quad}
\def\s{{\bf s}}
\def\lcm{\textup{lcm}}

\def\B{{\mathcal B}}
\def\M{{\bf M}}

\newcommand*{\bfrac}[2]{\begin{pmatrix} #1 \\ #2 \end{pmatrix}}
\def\Shift{\textup{Shift}}

\def\Ost{\textup{Ost}}

\def\C{\textbf{C}}
\def\Conv{\textbf{C}_{\for}}
\def\ConvK{\textbf{D}^{K}_{\for}}
\def\ConvM{\textbf{D}^{M}_{\for}}
\def\After{\textbf{After}}

\def\ZeroOne{\textup{\bf ZeroOne}_{\for\ex}}
\def\Pref{\textbf{Pref}_{\for\ex}}
\def\Read{\textup{\bf Read}_{\ex}}
\def\Next{\textup{\bf Next}_{\ex}}
\def\Tran{\textbf{Tran}_{\ex}}
\def\End{\textbf{E}_{\ex\for}}
\def\Alph{{\bf \Sigma}}
\def\States{{\bf Q}}

\def\Comp{\textbf{Compatible}}
\def\Cons{\textbf{Consec}}

\def\Best{\textbf{\bf Best}}
\def\Adm{\textbf{\bf Admissible}}
\def\Member{\textbf{\bf Member}}

\newcommand{\encode}[1]{\langle #1 \rangle}
\def\M{\mathcal{M}}
\def\Tape{\mathcal{T}}
\def\Transcript{{\bm T}}
\def\xsymb{{\bm \times}}

\newcommand\restr[3]{{
  \left.\kern-\nulldelimiterspace 
  #1 
  \vphantom{\big|} 
  \right|_{#2}^{#3} 
  }}

\def\wedge{\; \land \;}
\def\vee{\; \lor \;}

\title[Presburger Arithmetic with algebraic scalar multiplications]{Presburger Arithmetic with algebraic \\
scalar multiplications}

\author[P.~Hieronymi]{Philipp Hieronymi\rsuper{a}}
\address{\lsuper{a}Department of Mathematics\\
University of Illinois at Urbana-Champaign\\ 
Urbana, IL, 61801\\
USA}
\email{phierony@illinois.edu}

\author[D.~Nguyen]{Danny Nguyen\rsuper{b}}

\author[I.~Pak]{Igor Pak\rsuper{b}}
\address{\lsuper{b}Department of Mathematics\\
UCLA
\\ Los Angeles, CA, 90095\\ 
USA
}
\email{ldnguyen@math.ucla.edu}
\email{pak@math.ucla.edu}

\begin{document}

\maketitle

\begin{abstract}
We consider Presburger arithmetic (PA) extended by scalar multiplication by an algebraic irrational number $\alpha$, and call this extension $\alpha$-Presburger arithmetic ($\alpha$-PA). We show that the complexity of deciding sentences in $\alpha$-PA is substantially harder than in PA. Indeed, when $\al$ is quadratic and $r\geq 4$, deciding $ \alpha$-PA sentences
with $r$ alternating quantifier blocks and at most \ts $c\ts\ts r$ \ts
variables and inequalities requires space at least

$$K\. 2^{2^{ \, \iddots^{ \; 2^{ \, C \ell(S) } }}} \quad
\text{\rm $\bigl(\ts$tower of height \.\ts $r-3\ts\bigr)$},
$$
where the constants \ts $c, \ts K,\ts C>0$ \ts only depend on~$\alpha$, and $\ell(S)$ is the length of the given $\al$-PA sentence $S$. Furthermore deciding \ts $\ex^{6}\for^{4}\ex^{11}$ \ts $\alpha$-PA sentences with at most  $k$ inequalities is $\pspace$-hard, where $k$ is another constant depending only on~$\al$. When $\al$ is non-quadratic, already  four alternating quantifier blocks suffice for undecidability of $\alpha$-PA sentences.
\end{abstract}

\section{Introduction}

\subsection{Main results}
Let $\alpha$ be a real number. An \emph{$\alpha$-Presburger  sentence} (short: an $\al$-PA sentence) is a statement of the form
\begin{equation}\label{eq:sentence}
Q_{1} \x_{1} \in \zz^{n_{1}} \; \dots \; Q_{r} \x_{r} \in \zz^{n_{r}} \;\; \Phi(\x_{1},\dots,\x_{r}),
\end{equation}
where $Q_{1},\dots,Q_{r} \in \{\for,\ex\}$ are $r$ \emph{alternating quantifiers}, $\x_{1},\dots,\x_{r}$ are $r$ blocks of integer variables,
and $\Phi$ is a Boolean combination of linear inequalities in $\x_{1},\dots,\x_{r}$
with coefficients and constant terms in $\Z[\al]$.
As the number~$r$ of alternating quantifier blocks and the dimensions 
$n_1,\ldots,n_r$ increase, the truth of $\al$-PA sentences becomes harder to decide. 
In this paper, we study the computational complexity of deciding $\alpha$-PA 
sentences.

\bigskip

\noindent Sentences of the form ~\eqref{eq:sentence} have nice geometric interpretations in many
special cases.
Assume the formula $\Phi$ is a conjunction of linear equations
and inequalities, then $\Phi$ defines a convex polyhedron~$P$ defined over $\qqq[\al]$.  For $r=1$ and $Q_1=\exists$,
a sentence~of the form ~\eqref{eq:sentence} asks for existence of an integer point in~$P$:
\begin{equation}\label{eq:sentence1}
\exists\. \x \in \. P \cap \zz^{n}.
\end{equation}
In a special case of $r=2$,
$Q_1=\forall$ and $Q_2=\exists$, the sentence~$S$ can ask whether
projections of integer points in a convex polyhedron $P\subseteq \rr^{k+m}$
cover all integer points in another polyhedron $R\subseteq \rr^{k}$:
\begin{equation}\label{eq:sentence2}
\forall\. \x \in R \cap \zz^k  \; \; \exists \. \y \in  \zz^{m} \;\; : \;\; (\x,\y) \in P\ts.
\end{equation}
Here both $P$ and $R$ are defined over $\Q[\al]$.
When $\al$ is rational, the classical problems~\eqref{eq:sentence1} and~\eqref{eq:sentence2} are repectively known as \emph{Integer Programming} and \emph{Parametric Integer Programming} (see eg.~\cite{Schrijver},~\cite{L},~\cite{K1}).
Further variations on the theme and increasing number of quantifiers
allow more general formulas with integer valuations of the polytope algebra. For a survey of this area, see Barvinok~\cite{B2}.

\bigskip

\noindent Recall that classical Presburger arithmetic ($\text{PA}$) is the first-order theory of $(\Z,<,+)$, introduced by Presburger in~\cite{Pres}. When $\alpha\in \Q$, then deciding the truth of $\alpha$-PA sentence is equivalent to deciding whether or not a PA sentence is true. The latter decision problem has been studied extensively and we review some of their results below. The focus of this paper is the case when $\alpha$ is irrational, which is implicitly assumed whenever we mention $\alpha$-PA.

\bigskip

\noindent Let $\al \in \Qrc$, where $\Qrc$ is the field of real algebraic numbers. We think of $\al \in \Qrc$ as being given by its defining $\Z[x]$-polynomial of
degree $d$, with a rational interval to single out a unique root.
We say that $\al\in \Qrc$ is \emph{quadratic} if $d=2$.
Similarly, the elements $\ga \in \Z[\al]$ are represented in the form \ts
$\ga = c_{0} + c_{1}\al + \ldots + c_{d-1}\al^{d-1}$, where \ts $c_{0},\ldots,c_{d-1} \in \Z$.
For example, $\al=\sqrt{2}$ is quadratic and given by \ts
$\{\al^2-2=0, \al>0\}$. Thus \ts $\Z[\sqrt{2}]= \{a + b\ts\sqrt{2}, \, a,b\in \Z\}$. For $\ga \in \Z[\al]$, the \emph{encoding length} $\ell(\ga)$ is the
total bit length of the $c_i$'s defined above.
Similarly, the encoding length $\ell(S)$ of an $\alpha$-PA sentence $S$ is defined to be the total bit length
of all symbols in~$S$, with integer coefficients and constants represented in binary.

\bigskip

\noindent The only existing result that directly relates to the complexity of deciding $\alpha$-PA sentences is the following theorem due to Khachiyan and Porkolab, which extends Lenstra's classical result~\cite{L} on Integer Programming in fixed dimensions.

\begin{thmC}[\cite{KP}]\label{t:alg-IP}
  For every fixed $n$, sentences of the form
  $\ex\y \in \zz^{n} \, : \, A\y \le \b$ with $A \in \Qrc^{m \times n}, \b \in \Qrc^{m}$
  can be decided in polynomial time.
\end{thmC}

\noindent Note that the system \ts $A \y \le \cj b$ \ts in Theorem \ref{t:alg-IP} can involve arbitrary algebraic irrationals. This is a rare positive result on irrational polyhedra. Indeed,  for a non-quadratic~$\al$, this gives the only positive result
on deciding $\alpha$-PA sentences that we know of.

\bigskip

\noindent In this paper we establish that Theorem \ref{t:alg-IP} is an exception and in general deciding $\alpha$-PA sentences is often substantially harder. We first consider the case of $\alpha$ being quadratic. In this situation the theory $\al$-PA is decidable, and in this paper we prove both lower and upper bounds on the complexity of deciding the truth of a given $\al$-PA sentence. In the following theorems, the constants $K,C$ vary from one context to another.

\begin{thm}\label{th:quad-upper}
Let $\al\in \Qrc$ be a quadratic irrational number,
and let $r\ge 1$. An $\alpha$-\textup{PA} sentence $S$ with~$r$
alternating quantifier blocks can be decided in time at most
$$K\. 2^{2^{ \, \iddots^{ \; 2^{ \, C \ts \ell(S) } }}}  \quad
\text{\rm $\bigl(\ts$tower of height \. $r\ts\bigr)$},
 $$
where the constants \ts $K,\ts C>0$ \ts
depend only on $\alpha$.
\end{thm}

\noindent In the opposite direction, we have the following lower bound:

\begin{thm}\label{th:quad_lower}
Let $\al\in \Qrc$ be a quadratic irrational number,
and let $r\ge 4$.  Then deciding $\alpha$-\textup{PA} sentences
with $r$ alternating quantifier blocks and at most \ts $c\ts\ts r$ \ts
variables and inequalities requires space at least:
$$K\. 2^{2^{ \, \iddots^{ \; 2^{ \, C \ts \ell(S) } }}}  \quad
\text{\rm $\bigl(\ts$tower of height \.\ts $r-3\ts\bigr)$},
$$
where the constants \ts $c, \ts K,\ts C>0$ \ts only depend on~$\alpha$.
\end{thm}

\noindent These results should be compared with the triply exponential
upper bound and doubly exponential lower bounds for~PA (discussed in Section~\ref{sec:prev}).
In the borderline case of  $r=3$, the problem is especially
interesting.
We give the following lower bound, which only needs a few variables:

\begin{thm}\label{th:quad-pspace}
Let $\al\in \Qrc$ be a quadratic irrational number.
Then deciding \ts $\ex^{6}\for^{4}\ex^{11}$ \ts $\alpha$-\textup{PA} sentences with at most $K$ inequalities is $\pspace$-hard,
where the constant $K$ depends only on~$\al$.  Furthermore,
for $\al=\sqrt{2}$, one can take $K=10^6$.
\end{thm}

\nin
This should be compared with Gr\"{a}del's theorem on
$\SigmaP_{2}$-completeness for $\ex^{*}\for^{*}\ex^{*}$ integer sentences
in~PA (also discussed in Section~\ref{sec:prev}). The sudden jump from the polynomial hierarchy in PA is due to the power of irrational quadratics.
Specifically, any irrational quadratic $\al$ has an infinite periodic continued fraction, and this allows us to work with Ostrowski representations of integers in base $\al$. Because of this we are able to code string relations such as shifts, suffix/prefix and subset, which were not at all possible to define in PA.
Such operations are rich enough to encode arbitrary automata computation, and in fact Turing Machine computation in bounded space. Section~\ref{sec:outline} further discusses our method.

\bigskip

\noindent The situation is even worse when $\alpha$ is a non-quadratic irrational number. In this case $\al$-PA is undecidable and we prove here that just four alternating quantifier blocks are enough.

\begin{thm}\label{th:nonquad-undec}
  Let $\al \in \Qrc$ be a non-quadratic irrational number.
  Then \ts $\ex^{k}\for^{k}\ex^{k}\for^{k}$ $\alpha$-\textup{PA} sentences are undecidable, where $k=20000$.
\end{thm}

\subsection{Previous results}\label{sec:prev}
In this paper, we build on earlier related works on certain expansions of the real ordered additive group.
Let $\mathcal{S}_{\al} \coloneqq (\R, <, +,\zz, x \mapsto \al x)$. Denote by $T_{\al}$ the first-order theory of $\mathcal{S}_{\al}$; this is the first-order theory of the real numbers viewed as a $\qqq(\alpha)$-vector space with a predicate for the set of the integers. This theory is an extension of Presburger Arithmetic, and it is not hard to see that $T_{\al}$ contains an $\alpha$-PA sentence if only if the $\alpha$-PA sentence is true. 

\bigskip

\noindent
It has long been known that $T_{\al}$ is decidable when $\al$ is rational. This is arguably due to Skolem \cite{skolem} and was later
rediscovered independently by Weispfenning \cite{weis} and Miller~\cite{ivp}. More recently, Hieronymi~\cite{H} proved that $T_{\alpha}$ is still decidable when
$\al$ is quadratic. Thus it was known that checking $\alpha$-PA sentences is decidable for quadratic $\alpha$. 

\bigskip

\noindent
When $\alpha$ is irrational non-quadratic, the theory $T_{\alpha}$ is undecidable by Hieronymi and Tychonievich~\cite{HT-Proj}. This result itself does not imply undecidability of $\alpha$-PA sentences in this case, as $T_{\alpha}$ is a proper extension of $\alpha$-PA. However, the undecidability of $\alpha$-PA can be obtained by a careful analysis of the proof in ~\cite{HT-Proj}. In Theorem~\ref{th:nonquad-undec}, we not only give an explicit proof of this result, but also precisely quantify this result by showing that four alternating quantifier blocks are enough for undecidability. While our argument is based on the ideas in~\cite{HT-Proj}, substantial extra work is necessary to reduce the number of alternations to four from the lower bound implicit in the proofs in \cite{HT-Proj}.

\bigskip

\noindent While this paper is the first systematic study of the complexity of decision problems related to $\al$-PA, there is a large body of work for Presburger arithmetic. A quantifier elimination algorithm for PA was given by
Cooper~\cite{C} to effectively solve the decision problem.
Oppen~\cite{Oppen} showed that such sentences can be decided in at most triply
exponential time (see also~\cite{RL}).  In the opposite direction,
a nondeterministic doubly exponential lower bound was obtained
by Fischer and Rabin~\cite{FR} (see also~\cite{Wei}). We also refer the reader to Berman~\cite{Ber} for the precise complexity of Presburger arithmetic.
As one restricts the number of alternations, the complexity of PA drops by roughly one exponent (see~\cite{Fur,Sca,RL}), but still remains exponential.

\bigskip

\noindent For a bounded number of variables, two important cases are known to be polynomial time decidable,
namely the analogues of~\eqref{eq:sentence1} and~\eqref{eq:sentence2} with rational
polyhedra $P$ and~$R$.  These are classical results by Lenstra~\cite{L}
and Kannan~\cite{K1}, respectively.  Scarpellini~\cite{Sca} showed that all
$\ex^{n}$-sentences are still polynomial time decidable for every $n$ fixed.
However, for two alternating quantifiers, Sch\"{o}ning proved in~\cite{Sch} that
deciding $\ts\exists y \ts \forall x\ts:\ts\Phi(x,y)$ is \NP-complete.
Here $\Phi$ is any Boolean combination of linear inequalities in two variables,
instead of those in the particular form~\eqref{eq:sentence2}.
This improved an earlier result by Gr\"{a}del in~\cite{Gra}, who also showed
that PA sentences with $m+1$ alternating quantifier blocks and $m+5$ variables
are complete for the $m$-th level in the Polynomial Hierarchy~\polyH.
In these results, the number of inequalities (atoms) in $\Phi$ is still part of the input, i.e., allowed to vary.

\bigskip

\noindent Much of the recent work concerns the most restricted PA sentences for which the number of alternations ($r+2$),  number of variables \emph{and} number
of inequalities in $\Phi$ are all fixed.
Thus, the input is essentially a bounded list of integer coefficients
and constants in~$\Phi$, encoded in binary.
For $r=0$, such sentences are polynomial time decidable by Woods~\cite{Woods}.
For $r=1$, Nguyen and Pak~\cite{NP} showed that deciding \ts $\exists^1\forall^2\exists^2$ \ts
PA-sentences with at most $10$ inequalities is \NP-complete.
More generally, they showed that such sentences with $r+2$ alternations, $O(r)$ variables and inequalities are complete for the $r$-th level in~\polyH.
Thus, limiting the ``format'' of a
PA formula does not reduce the complexity by a lot.  This is our main motivation
for the lower bounds in Theorems~\ref{th:quad_lower} and~\ref{th:quad-pspace}
for $\alpha$-PA sentences.

\subsection{Proofs outline}\label{sec:outline}
Let \textup{S1S} be the monadic second order theory of $(\N,+1)$, where $+1$ denotes the usual successor function, and let \textup{WS1S} be the weak monadic second order theory of $(\N,+1)$, that is the monadic second order logic of $(\N,+1)$ in which quantification over sets is restricted to quantification over finite subsets. The main results of \cite{H} states that for quadratic $\alpha$, one can decide $T_{\alpha}$-sentences by translating them into corresponding \textup{S1S}-sentences, and then decide the latter. Since $\alpha$-PA sentences form a subset of all $T_{\alpha}$-sentences, this method can be used to decide $\alpha$-PA sentences. Thus, upper complexity bounds for \textup{S1S} can theoretically be transferred to deciding $\alpha$-PA sentences. Moreover, the work in \cite{H} also shows that one can translate \textup{S1S}-sentences into $T_{\alpha}$-sentences. However, no efficient direct translation between $\cal{L}_{\alpha}$-sentence and \textup{S1S}-sentence was given in~\cite{H2,H}.
Ideally, one would like to do this translation with as few extra alternations of quantifiers as possible.
In Theorems~\ref{th:quad-upper} and~\ref{th:quad_lower}, we explicitly quantify this translation. We strengthen the result from \cite{H} by showing that one can translate $\al$-PA sentences to \textup{WS1S}-sentences. The translation then allows to us find upper and lower complexity bounds for deciding $\alpha$-PA sentences.

\bigskip

\noindent The most powerful feature of $\alpha$-PA sentences is that we can talk about Ostrowski representation of integers, which will be used throughout the paper as the main encoding tool. We first obtain the upper bound in Theorem~\ref{th:quad-upper} by directly translating $\al$-PA sentences into the statements about automata using Ostrowski encoding and using known upper bounds for certain decision problems about automata. Next, we show the lower bound for three alternating quantifiers
(Theorem~\ref{th:quad-pspace}) by a general argument on the
Halting Problem with polynomial space constraint, again using Ostrowski encoding. We generalize the above argument to get a lower bound for $r \ge 3$ alternating quantifier blocks (Theorem~\ref{th:quad_lower}). For the latter result, we first translate \textup{WS1S}-sentences to $\alpha$-PA sentences with only one extra alternation,
and then invoke a known tower lower bound for \textup{WS1S}. Finally in the proof of Theorem \ref{th:nonquad-undec}, we again use the expressibility of Ostrowski representation to reduce the upper bound of the number of  alternating quantifier blocks needed for undecidability of $\alpha$-PA sentences. The use of Ostrowski representations allows us to replace more general arguments from \cite{HT-Proj} by explicit computations, and thereby reduce the quantifier-complexity of certain $\alpha$-PA sentences.

\subsection*{Notation}
We use bold notation like $\x,\y$ to indicate vectors of variables.

\section{Continued fractions and Ostrowski representation}\label{sec:ctn_frac}

Ostrowski representation and continued fractions play a crucial role throughout this paper. 
We recall basic definitions and facts in this subsection. We refer the reader to 
Rockett and Sz\"usz~\cite{RS} for more details and proofs.

\medskip

\noindent A \emph{finite continued fraction expansion} $[a_0;a_1,\dots,a_k]$ is an expression of the form
\[
a_0 + \frac{1}{a_1 + \frac{1}{a_2+ \frac{1}{\ddots +  \frac{1}{a_k}}}}
\]
For a real number $\alpha$, we say $[a_0;a_1,\dots,a_k,\dots]$ is \emph{the continued fraction expansion of $\alpha$} if $\alpha=\lim_{k\to \infty}[a_0;a_1,\dots,a_k]$ and $a_0\in \Z$, $a_i\in \N_{>0}$ for $i>0$. For the rest of this subsection, fix a positive irrational real number $\alpha$ and let $[a_0;a_1,a_2,\dots]$ be the continued fraction expansion of $\alpha$.

\medskip

\noindent Let $k\geq 1$. A quotient $p_k/q_k \in \Q$ is said to be the \emph{$k$-th convergent of $\alpha$} if $p_k\in \N$, $q_k\in \Z$, $\gcd(p_k,q_k)=1$ and
\[
\frac{p_k}{q_k} \. = \. [a_0;a_1,\dots,a_k]\ts.
\]
It is well known that the convergents of $\al$ follow the recurrence relation:
\begin{equation}\label{eq:rec}
\gathered
(p_{-1},q_{-1}) \. = \. (1,0);\; (p_{0},q_{0}) \. = \. (a_{0},1);\; \\
p_{n} = a_{n}p_{n-1} + p_{n-2},\; q_{n} = a_{n}q_{n-1} + q_{n-2} \quad\text{for } n \ge 1.
\endgathered
\end{equation}
This can be written as:
\begin{equation}\label{eq:matrix_rec}
 \begin{pmatrix}
    p_{n} & p_{n-1} \\
    q_{n} & q_{n-1}
 \end{pmatrix} = \Gamma_{0} \cdots \Gamma_{n} \  , \ \quad
 \text{where} \quad \Gamma_{i} =
 \begin{pmatrix}
  a_{i} & 1\\
  1 & 0
\end{pmatrix}.
\end{equation}

\begin{factC}[{\cite[Chapter II.2 Theorem 2]{RS}}]\label{bestrational} The set of best rational approximations of $\alpha$ is precisely the set of all convergents $\{p_{k}/q_{k}\}$ of $\alpha$.
In other words, for every $p_{k}/q_{k}$, we have:
\begin{equation}\label{eq:conv1}
\for x,y \in \Z \;\; (0 < y < q_k) \to |y \al -x| >  |q_k \al - p_{k}|.
\end{equation}
\end{factC}

\noindent
The \emph{$k$-th difference of $\alpha$} is defined as $\beta_k := q_k \alpha - p_k$. We use the following properties of the $k$-th difference:
\begin{equation}\label{eq:odd_even}
\be_{n} > 0 \;\;\text{if}\;\; 2|n,\quad \be_{n} < 0 \;\;\text{if}\;\; 2\nmid n.
\end{equation}
\begin{equation}\label{eq:beta_dec}
\be_{0} \; > \; -\be_{1} \; > \; \be_{2} \; > \; -\be_{3} \; > \; \dots
\end{equation}
\begin{equation}\label{eq:sum_beta}
-\be_{n} \; = \; a_{n+2} \be_{n+1} + a_{n+4} \be_{n+3} + a_{n+6} \be_{n+5} + \dots \quad \for n \in \N.
\end{equation}
These can be easily proved using~\eqref{eq:rec}. We now introduce a class of numeration systems introduced by Ostrowski \cite{Ost}. 

\begin{fact}\label{f:Ostrowski}
Let $X \in \N$. Then $X$ can be written uniquely as
\begin{equation}\label{eq:Ost}
X = \sum_{n=0}^{N}b_{n+1}q_{n}.
\end{equation}
where $0 \le b_{1}<a_{1}$, $0 \le b_{n+1} \le a_{n+1}$ and $b_{n}=0$ whenever $b_{n+1}=a_{n+1}$.
\end{fact}
\begin{proof}
See~\cite[Ch.~II-\S4]{RS}.
\end{proof}
\nin
We refer to~\eqref{eq:Ost} as the \emph{$\al$-Ostrowski representation of $X$}. When $\alpha$ is clear from the context, we simply say the Ostrowski representation of $X$.
We also denote the coefficients $b_{n+1}$ in~\eqref{eq:Ost} by $[q_{n}](X)$. When $X$ is obvious from the context, we just write $[q_n]$. We denote by $\Ost(X)$ the set of $q_{n}$ with $[q_{n}](X) > 0$.

\medskip

\noindent 
Observe that $a_0 -\alpha \in (-1,0)$.
Let $I_{\alpha}$ be the interval $\bigl[a_0-\alpha,1 + (a_0 - \alpha)\bigr)$.
  Define $f_{\alpha} : \N \to I_{\al}$ to be the function that maps $X$ to $\alpha X - U$, where $U$ is the unique natural number such that $\alpha X - U \in I_{\alpha}$.
In other words:
\begin{equation}\label{eq:f}
f_{\alpha}(X)= \alpha X - U \ \Longleftrightarrow \ a_0-\alpha\leq \alpha X - U <1 + (a_0- \alpha)\..
\end{equation}
Let $g_{\alpha} : \N \to \N$ be the function that maps $X$ to the natural number $U$ satisfying $\alpha X - U\in I_{\al}$. The reader can check that $\alpha X = f_{\alpha}(X) + g_{\alpha}(X)$.

\begin{fact}\label{fact:f}
  Let $X\in \N$. Then
\begin{equation}\label{eq:snap_rem}
  f_{\alpha}(X) = \sum_{n=0}b_{n+1} \be_{n} \qquad \text{and} \qquad g_{\alpha}(X) = \sum_{n=0}b_{n+1}p_{n},
\end{equation}
where the coefficients $b_{n}$ are from~\eqref{eq:Ost}, and $f_{\alpha}(\N) = \{f_{\alpha}(X) \,:\, X \in \N\}$ is a dense subset of the interval $[-\frac{1}{\zeta_{\alpha}}, 1 - \frac{1}{\zeta_{\alpha}})$.
\end{fact}
\begin{proof}
See ~\cite[Th.~1 on p.~25]{RS} and~\cite[Th.~1 on p.~33]{RS}.
\end{proof}

\subsection{Periodic continued fractions}\label{sec:periodic} Lagrange's theorem states that
an irrational number $\alpha$  is  quadratic  if and only if it has a periodic continued fraction expansion (see \cite[Chapter III.1 Theorem 1 and 2]{RS}). In this situation, we write
\[
\al = [a_{0};a_{1}, \dots, a_{m}, \cj{c_{0},c_{1},\dots,c_{k_{\alpha}-1}}],
\]
where $c_0,\dots,c_{k_{\alpha}-1}$ is the repeating block in the continued fraction expansion  with the \emph{minimum} period $k_{\alpha}$.

\medskip

\noindent Let $\gamma = [\cj{c_{0}; c_{1},\dots,c_{k-1}}]$.
It is easy to see that $\gamma = (c\al + d)/(e\al + f)$ for some $c,d,e,f \in \zz$.
Therefore, $\alpha$-PAs sentences can be expressed in $\gamma$-PA sentences and vice versa. Since $\alpha$ is constant, so are $c,d,e$ and $f$. Thus, for our complexity purposes, we can always assume that our quadratic irrational $\al$ is purely periodic, i.e.,
\begin{equation}\label{eq:periodic}
\al = [\cj{a_{0}; a_{1}, \dots, a_{k_{\alpha}-1}}]
\end{equation}
  with the minimum period $k_{\alpha}$.

\begin{fact}\label{fact:div_k}
  Let $i \in \N$.
  There exist $c_{i},d_{i} \in \Z$ such that for every $n \in \N$ with $k_{\alpha} | n$, we have:
  $$(p_{n+i},q_{n+i}) \;=\; c_{i}(p_{n},q_{n}) + d_{i}(p_{n+1},q_{n+1}).$$
  The coefficients $c_{i},d_{i}$ can be computed in time $\polyin(i)$.
\end{fact}
\begin{proof}
By~\eqref{eq:matrix_rec},  we have:
  $$
\begin{pmatrix}
p_{n+i+1} & p_{n+i} \\ q_{n+i+1} & q_{n+i}
\end{pmatrix}
\; = \;
\Gamma_{0}\, \cdots\, \Gamma_{n+1} \,\, \Gamma_{n+2} \cdots \Gamma_{n+i+1}
\; = \;
\begin{pmatrix}
p_{n+1} & p_{n} \\ q_{n+1} & q_{n}
\end{pmatrix}
\, \Gamma_{n+2} \cdots \Gamma_{n+i+1}\ts.
  $$
Since \. $\Gamma_{k+t}=\Gamma_{t}$ \. for every $t \in \N$ and $k_{\alpha} | n$, 
we have \, $\Gamma_{n+2} \dots \Gamma_{n+i+1} \. = \. \Gamma_{2} \cdots \Gamma_{i+1}$.
Let \ts $c_i,d_i,c_i',d_i' \in \Z$ \ts be such that
\begin{equation}\label{eq:cd_rec}
\begin{pmatrix}
d'_{i} & d_{i} \\ c'_{i} & c_{i}
\end{pmatrix} \. = \. \Gamma_{2}\,\cdots\,\Gamma_{i+1}\ts .
\end{equation}
This choice immediately gives that
\[
\begin{pmatrix}
p_{n+i+1} & p_{n+i} \\ q_{n+i+1} & q_{n+i}
\end{pmatrix}
\; = \;
\begin{pmatrix}
p_{n+1} & p_{n} \\ q_{n+1} & q_{n}
\end{pmatrix}
\begin{pmatrix}
d'_{i} & d_{i} \\ c'_{i} & c_{i}
\end{pmatrix}.
\]
Thus $(p_{n+i},q_{n+i}) \;=\; c_{i}(p_{n},q_{n}) + d_{i}(p_{n+1},q_{n+1})$. Note that $c_{i},d_{i}$ only depend on $i$ and can be computed in time $\polyin(i)$ by~\eqref{eq:cd_rec}.
\end{proof}

\begin{rem}\label{rem:gamma}
  For $i=0$, we have $c_{0}=1,d_{0}=0$.
  For $i=1$, we have $c_{1}=0,d_{1}=1$.
  By~\eqref{eq:cd_rec},
  if we let $\gamma_{i}(v,v') \coloneqq c_{i} v + d_{i} v'$ then they follow the recurrence:
\begin{equation}\label{eq:gamma_rec}
\gamma_{0}(v,v') = v,\; \gamma_{1}(v,v') = v',\; \gamma_{i}(v,v')=a_{i}\gamma_{i-1}(v,v')+\gamma_{i-2}(v,v'),
\end{equation}
as similar to~\eqref{eq:rec}.
\end{rem}

\begin{fact}\label{fact:nice_rec}
There are $\mu,\nu,\mu',\nu' \in \Q$ such that for all $n\in \N$
\[
p_{n} = \mu q_{n} + \mu' q_{n+k_{\alpha}}, \quad q_{n} = \nu p_{n} + \nu' p_{n+k_{\alpha}}
\]
\end{fact}

\begin{proof}

From~\eqref{eq:matrix_rec}, for every $n \in \N$ we have: 
$$\begin{pmatrix}
p_{n} \\ q_{n}
\end{pmatrix}
\; = \;
\Gamma_{0} \, \Gamma_{1} \, \cdots \, \Gamma_{n} \,
\begin{pmatrix}
1 \\ 0
\end{pmatrix}.
$$
Since $\Gamma_{i+k_{\alpha}}=\Gamma_{i}$,
$$
\aligned
\bfrac{p_{n+k_{\alpha}}}{q_{n+k_{\alpha}}} \, &= \,  \bigl(\Gamma_{0} \dots \Gamma_{k_{\alpha}-1}\bigr) \; (\Gamma_{k_{\alpha}} \cdots \Gamma_{n+k_{\alpha}}) \bfrac{1}{0}
\, = \, \bigl(\Gamma_{0} \cdots \Gamma_{k_{\alpha}-1}\bigr) \; \bigl(\Gamma_{0} \cdots \Gamma_{n}\bigr) \bfrac{1}{0}  \\
& = \; \bigl(\Gamma_{0} \cdots \Gamma_{k_{\alpha}-1}\bigr) \bfrac{p_{n}}{q_{n}}
\; = \;
\begin{pmatrix}
p_{k_{\alpha}-1} & p_{k_{\alpha}-2}\\
q_{k_{\alpha}-1} & q_{k_{\alpha}-2}
\end{pmatrix}
\bfrac{p_{n}}{q_{n}}
.
\endaligned
$$
Note that $p_{k_{\alpha}-1},q_{k_{\alpha}-1},p_{k_{\alpha}-1},q_{k_{\alpha}-2}$ do not depend on $n$. From here we easily get $\mu,\nu,\mu',\nu'$ with the desired property.
\end{proof}


\bigskip

\section{$\al$-Presburger formulas}\label{section:alphapa}
Fix some $\alpha \in \R$. An \emph{$\alpha$-PA formula} is of the form
\[
Q_{1} \y_{1} \; \dots \; Q_{r} \y_{r} \;\; \Phi(\y_{1},\dots,\y_{r},\x),
\]
where $\Phi$ is a Boolean combination of linear inequalities in $\y_{1} \in \Z^{n_{1}},\dots,\y_{r}\in\Z^{n_{r}}, \x\in\Z^{m}$ with coefficients and constant terms in $\Z[\al]$ and $\y_{1},\dots,\y_{r},\mathbf{x}$ are integer variables; or any logically equivalent first-order formula in the language $\mathcal{L}_{\alpha}=\{+,0,1,<,\lambda_{p} : p \in \Z[\al]\}$ where $\lambda_{p}$ is a unary function symbol for multiplication by $p \in \Z[\alpha]$. We will denote a generic $\al$-PA formula as $F(\x)$, where $\x$ are the free variables of $F$, i.e., those not associated with a quantifier. An \emph{$\alpha$-PA sentence} is an $\alpha$-PA formula without free variables.

\medskip

\noindent Given an $\al$-PA formula $F(\x)$ and $\mathbf{X}\in \Z^{|\x|}$, we say $F(\mathbf{X})$ holds (or is true) if the statement obtained by replacing the free variables in $F$ by $\mathbf{X}$ and letting the quantified variables $\y_{i}$ range over $\Z^{n_{i}}$, is true.
We say that a set $S\subseteq \Z^m$ is \emph{$\al$-PA definable} (or an \emph{$\al$-PA set}) if there exists an $\al$-PA formula $F(\x)$ such that
\[
S = \{ \mathbf{X} \in \Z^{|\x|} \ : \ F(\mathbf{X})\}.
\]

\noindent When $\alpha=0$, then a $\al$-PA formula is just a classical  PA-formula. Hence $0$-PA is just PA, and therefore decidable.  Let $\mathcal{S}_{\al} = (\R, <, +,\zz, x \mapsto \al x)$. As pointed out in the introduction,  the first-order theory $T_{\al}$ of $\mathcal{S}_{\al}$ contains all true $\alpha$-PA sentence. Since $T_{\al}$ is decidable by \cite{H}, we have:

\begin{thm} Let $\alpha$ be quadratic. Then $\alpha$-\textup{PA} is decidable.
\end{thm}

\noindent The main difference between the situation when $\alpha$ is rational and when it is irrational, is that when $\alpha$ is irrational, $\alpha$-PA formulas can express properties of the $\alpha$-Ostrowski representation of natural numbers. This increases the computational complexity of the decision procedure of $\alpha$-PA in comparison to the one for PA.

\subsection{$\al$-PA formulas for working with Ostrowski representation}\label{sec:Ostrowski}
Let $\al$ be an irrational number, not necessarily quadratic. In this section, we will show that various properties of Ostrowski representations can be expressed using $\alpha$-PA formulas.

\medskip

\noindent By Fact \ref{bestrational} the convergents $\{p_{n}/q_{n}\}$ of $\alpha$ can be characterized by the best approximation property.
Namely, $u/v$ with $v > 1$ is a convergent $p_{n}/q_{n}$ for some $n\in \N$ if and only if
\begin{equation}\label{eq:conv2}
\for w,z \; (0 < z < v) \to |z \al  - w| >  |v\alpha  - u|.
\end{equation}
Here $\gcd(u,v)=1$ is implied, since if $k=\gcd(u,v)>1$, then $|\al v/k - u/k| < |\al v - u|$ and $0 < v/k < v$. Now consider two consecutive convergents $(u,v)=(p_{n},q_{n})$ and $(u',v')=(p_{n+1},q_{n+1})$ for some $n\in \N$. For any integers $0 < z < v'$ and $w$, first we have $|z \al - w| > |v' \al - u'|$. If $|z \al  - w| < |v \al - u|$, then first we must have $v < z < v'$. Then among all such pairs $(w,z)$, the one with the minimum $|z \al - w|$ must necessarily be another convergent of $\al$, which is impossible since we assumed that $(u,v)$ and $(u',v')$ are consecutive. Thus, a necessary and sufficient condition for $(u,v)$ and $(u',v')$ to be consecutive convergents is simply:
\begin{equation}\label{eq:consec}
\gathered
  \Conv(u,v,u',v') \; \coloneqq  \; 0 < v < v' \; \land \; \for w,z  \; \big( 0 < z < v' \to \\
 |z\al - w| \ge |v\al - u | > |v'\al - u'| \big).
\endgathered
\end{equation}
Note that $\Conv$ is a $\forall$-formula.
More generally, consider the $\al$-PA formula:
\begin{equation}\label{eq:multi_consec}
  \gathered
\Conv(u_{0},v_{0},\dots,u_{k},v_{k}) \; \coloneqq  \; 1 < v_{0} < v_{1} < \dots < v_{k} \; \land\\
\for w,z \; \bigwedge_{i=0}^{k} \Big( 0 < z < v_{i+1} \to |z \al - w| \ge |v_i \al - u_{i}| > |v_{i+1} \al - u_{i+1}| \Big).
\endgathered
\end{equation}
Then $\Conv$ is true if and only if $(u_{0},v_{0})=(p_{n},q_{n}), \dots, (u_{k},v_{k})=(p_{n+k},q_{n+k})$ for some $n$ with $q_{n}>1$, i.e., $k+1$ consecutive convergents of $\al$.

\medskip

\noindent Define the following quantifier-free $\alpha$-PA formulas:
\begin{equation}\label{eq:After}
\aligned
\After(u,v,u',v',Z,Z') \; \coloneqq \; & \bigl(-\al v + u < \al Z - Z' < -\al v' + u'\bigr)  \\
& \vee \bigl(-\al v' + u' < \al Z - Z' < -\al v + u\bigr).
\endaligned
\end{equation}
\begin{equation}\label{eq:wtAfter}
\aligned
\wt\After(u,v,u',v',Z,Z') \; \coloneqq  \; & \bigl(-\al v + u - \al v' + u' < \al Z - Z' < -\al v' + u'\bigr)\\
& \vee \bigl(-\al v' + u' < \al Z - Z' < -\al v + u - \al v' + u'\bigr).
\endaligned
\end{equation}

\begin{fact}\label{fact:After}
Let $u,v,u',v'\in \N$ and $n\in \N$ such that $(u,v)=(p_{n},q_{n})$ and $(u',v')=(p_{n+1},q_{n+1})$, and let $Z \in\N$. Then
\begin{enumerate}[label=(\roman*)]
\item $\Ost(Z) \subseteq \{q_{n+1}, q_{n+2}, \dots\}$ if and only if $\After(u,v,u',v',Z,Z')$ holds for some $Z'\in \N$.
  \item $\Ost(Z) \subseteq \{q_{n+1}, q_{n+2}, \dots\}$ and $[q_{n+1}](Z) < a_{n+2}$ if and only if $\wt\After(u,v,u',v',Z,Z')$ holds for some $Z'\in\N$.
\end{enumerate}
Also $Z'$ is uniquely determined by $Z$ if $\After$ or $\wt \After$ holds.
\end{fact}

\begin{proof}
This proof is similar to the proofs of Lemmas 4.6, 4.7 and 4.8 in ~\cite{H}.

\medskip

\noindent  i)
  Assume $n$ is odd.
  If $\Ost(Z) \subseteq \{q_{n+1}, q_{n+2}, \dots\}$, then the $\al$-Ostrowski representation of $Z$ is $Z = \sum_{k = n+1}^{N} b_{k+1} q_{k}$ for some $N \ge n+1$.
  From Fact~\ref{fact:f}, we have $f_{\al}(Z) = \sum_{k=n+1}^{N} b_{k+1} \be_{k}$.
  By~\eqref{eq:odd_even}, we have have $\be_{k} > 0$ if $k$ is even and $\be_{k} < 0$ if $k$ is odd.
Combining this with $b_{k+1} \le a_{k+1}$, we obtain
  $$
a_{n+3}\be_{n+2} + a_{n+5}\be_{n+4} + \dots \; < \; f_{\al}(Z) = \sum_{k=n+1}^{N} b_{k+1}\be_{k} \; < \; a_{n+2}\be_{n+1} + a_{n+4}\be_{n+3} + \dots
  $$
By~\eqref{eq:sum_beta}, this can be written as $-\be_{n+1} < f_{\al}(Z) < -\be_{n}$.
By~\eqref{eq:f}, we have $f_{\al}(Z) = \al Z - Z'$, where $Z' \in \N$ is unique such that $aZ - Z' \in I_{\alpha}$.
Also note that $\be_{n} = \al v - u$ and $\be_{n+1} = \al v' - u'$.
So the above inequalities can be written as $-\al v' + u' < \al Z - Z' < -\al v + u$.
When $n$ is even, the inequalities reverse to $-\al v + u < \al Z - Z' < -\al v' + u'$.
Thus $\Ost(Z) \subseteq \{q_{n+1}, q_{n+2}, \dots\}$ implies $\After(u,v,u',v',Z,Z')$.
The converse direction can be proved similarly, using~\eqref{eq:beta_dec} and~\eqref{eq:sum_beta}.

\medskip

\noindent ii) The only difference here is that $[q_{n+1}](Z) = b_{n+2}$ can be at most $a_{n+2}-1$. Details are left to the reader.
 \end{proof}

\noindent The relation $v \in \Ost(X)$, meaning that $v=q_{n}$ appears in $\Ost(X)$, is definable by an existential $\al$-PA formula:
\begin{equation}\label{eq:In_ex}
\ex Z_{1},Z_{2},Z_{3}  \;\; (v \le Z_{1} < v') \; \land \; \wt\After(u,v,u',v',Z_{2},Z_{3}) \; \land \; X = Z_{1} + Z_{2}
\end{equation}
  and by a universal $\al$-PA formula:
\begin{equation}\label{eq:In_for}
 \for Z_{1},Z_{2},Z_{3} \;\; \Big[ (Z_{1} < v) \; \land \; \After(u,v,u',v',Z_{2},Z_{3})  \Big] \to Z_{1} + Z_{2} \neq X\ts.
\end{equation}
  To see this, note that $v \notin \Ost(X)$ if and only if $X = Z_{1} + Z_{2}$ for some $Z_{1}, Z_{2}$ with  $\Ost(Z_{1}) \subseteq \{q_{0},q_{1},\dots,q_{n-1}\}$ and $\Ost(Z_{2}) \subseteq \{q_{n+1},q_{n+2},\dots\}$.

\medskip

\noindent We will need one more quantifier-free $\al$-PA formula:
\begin{equation}\label{eq:Compatible}
\aligned
\Comp(u,v,u',v',X,Z,Z') \; \coloneqq  \;  X < &\ v' \; \land \; \After(u,v,u',v',Z,Z')  \\
 & \land  \; \Big( X \ge v \to \wt\After(u,v,u',v',Z,Z') \Big).
\endaligned
\end{equation}

\begin{fact}\label{fact:Compatible}
Let $u,v,u',v'\in \N$ and $n\in \N$ such that $(u,v)=(p_{n},q_{n})$ and $(u',v')=(p_{n+1},q_{n+1})$, and let $X,Z \in \N$.
Then $\Comp(u,v,u',v',X,Z,Z')$ holds for some $Z'\in \N$ if and only if
\begin{itemize}
\item $\Ost(X) \subseteq \{q_{0},\dots,q_{n}\}$ (by $X < v'$),
\item $\Ost(Z) \subseteq \{q_{n+1},q_{n+2},\dots\}$ (by $\After$),
\item if $q_{n} \in \Ost(X)$, then $[q_{n+1}](Z)< a_{n+2}$   (by $\wt\After$).
\end{itemize}
\end{fact}
\noindent In other words, $\Comp$ is satisfied if and only if $\Ost(X)$ and $\Ost(Z)$ can be directly concatenated at the point $v=q_{n}$ to form $\Ost(X+Z)$ (see~\eqref{eq:Ost}).

\bigskip

\section{Quadratic irrationals: Upper bound}

In this section we prove Theorem~\ref{th:quad-upper}. Let $\alpha\in \R$ be an irrational quadratic number.
We will show that for every $\al$-PA formula $F(\x)$, there is a finite automaton $\mathcal{A}$ such that $\mathcal{A}$ accepts precisely those words that are Ostrowski representations of numbers satisfying $F$. This will then allow to use automata-based decision procedure to decide $\al$-PA sentence. It should be emphasized that the tower height in Theorem~\ref{th:quad-upper} only depends on the number of alternating quantifiers, but not on the number of variables in the sentence.

\subsection{Finite automata and Ostrowski representations}

\noindent We first remind the reader of the definitions of finite automata and recognizability. For more details, we refer the reader to Khoussainov and Nerode~\cite{automata}. Let $\Sigma$ be a finite set. We denote by $\Sigma^*$ the set of words of finite length on $\Sigma$.

\begin{defi} A \emph{nondeterministic finite automaton} (NFA) $\mathcal{A}$ over $\Sigma$ is a tuple $(S,I,T,F)$, where $S$ is a finite non-empty set, called the set of states of $\mathcal{A}$, $I$ is a subset of $S$, called the set of initial states, $T\subseteq S \times \Sigma \times S$ is a non-empty set, called the transition table of $\mathcal{A}$ and $F$ is a subset of $S$, called the set of final states of $\mathcal{A}$. An automaton $\mathcal{A}=(S,I,T,F)$ is \textbf{deterministic} (DFA) if $I$ contains exactly one element, and for every $s\in S$ and $\sigma \in \Sigma$ there is exactly one $s' \in S$ such that $(s,\sigma,s')\in T$.
We say that an automaton $\mathcal{A}$ on $\Sigma$ \textbf{accepts} a word $w=w_n\dots w_1 \in \Sigma^*$ if there is a sequence $s_n,\dots, s_1,s_0 \in S$ such that $s_n \in I$, $s_0 \in F$ and for $i=1,\dots,n$, $(s_i,w_i,s_{i-1})\in T$. A subset $L\subseteq \Sigma^*$ is \textbf{recognized} by $\mathcal{A}$ if $L$ is the set of $\Sigma$-words that are accepted by $\mathcal{A}$. We say that $L\subseteq \Sigma^*$ is \textbf{recognizable} if $L$ is recognized by some DFA.
\end{defi}

\noindent
By the \emph{size} of an automaton, we mean its number of states.
It is well known that recognizability by NFA and DFA are equivalent:

\begin{factC}[{\cite[Theorem 2.3.3]{automata}}]\label{fact:NFA_DFA}
If $L$ is recognized by an NFA of size $m$, then $L$ is recognized by a DFA of size $2^{m}$.
\end{factC}

\noindent Let $\Sigma$ be a set containing $0$. Let $z=(z_1,\dots,z_n) \in (\Sigma^*)^n$ and let $m$ be the maximal length of $z_1,\dots,z_n$. We add to each $z_i$ the necessary number of $0$'s to get a word $z_i'$ of length $m$. The \textbf{convolution} of $z$ is defined as the word $z_1 * \dots * z_n \in (\Sigma^n)^*$ whose $i$-th letter is the element of $\Sigma^n$ consisting of the $i$-th letters of  $z_1', \dots, z_n'$.

\begin{defi} A subset $X\subseteq (\Sigma^*)^n$ is called \textbf{$\Sigma$-recognizable} if the set
\[
\bigl\{z_1*\dots *z_n \ : \ (z_1,\dots,z_n) \in X\bigr\} \quad \text{is $\Sigma^n$-recognizable.}
\]
\end{defi}

\noindent $\Sigma$-recognizable sets are closed under Boolean operations and first order quantifiers:

\begin{factC}[{\cite[\S2.3]{automata}}]\label{fact:automata_Boolean}
  If $X_{1},X_{2} \subseteq (\Sigma^{*})^{n}$ are recognized by DFAs of size $m_{1}$ and $m_{2}$, respectively, then:
\begin{enumerate}[label=\alph*)]
\item $X_{1}^{c}$ is recognized by a DFA of size $m_{1}$.
\item $X_{1} \cap X_{2}$, $X_{1} \cup X_{2}$ by DFAs of size $m_{1}m_{2}$.
\end{enumerate}
\end{factC}

\begin{fact}\label{fact:automata_qe}
  If $Z\subseteq(\Sigma^{*})^{n_{1}+n_{2}}$ is recognized by an NFA of size $m$, then
  $$X = \{x \in (\Sigma^{*})^{n_{1}} : \ex y \in (\Sigma^{*})^{n_{2}} \; (x,y) \in Z\} \text{ and } X^{c} = \{x \in (\Sigma^{*})^{n_{1}} : \for y \in (\Sigma^{*})^{n_{2}} \; (x,y) \in Z^{c}\}$$
  are recognized by DFAs of size $2^{m}$.
\end{fact}
\begin{proof}
The set $X$ is recognized by an NFA of size $m$ (see \cite[Theorem 2.3.9]{automata}). Thus $X$ is recognized by an DFA of size $2^m$ by Fact~\ref{fact:NFA_DFA}. It follows from Fact~\ref{fact:automata_Boolean} that $X^c$ can be recognized by a DFA of size $2^m$.
\end{proof}

\medskip

\noindent Let $\alpha$ be a quadratic irrational. Since the continued fraction expansion  $[a_0;a_1,a_2,\dots]$ of $\alpha$ is periodic, it is bounded. Let $M\in \N$ be the maximum of the $a_i$'s. Set $\Sigma_{\alpha} := \{0,\dots,M\}$. Recall from Fact~\ref{f:Ostrowski} that every $N \in \N$ can be written uniquely as $N = \sum_{k=0}^{n} b_{k+1} q_{k}$ such that $b_n\neq 0$, $b_{k} \in \N$ such that $b_1<a_1$, $b_k \leq a_{k}$ and, if $b_k = a_{k}$, $b_{k-1} = 0$. We denote the $\Sigma_{\alpha}$-word $b_1\dots b_n$ by $\rho_{\alpha}(N)$.

\begin{defi} Let $X\subseteq \N^n$. We say that $X$ is \textbf{$\alpha$-recognized} by a finite automaton $\mathcal{A}$ over $\Sigma_{\alpha}^n$ if the set
\[
\{ \rho_{\alpha}(N_1)*\dots*\rho_{\alpha}(N_n) \ : \ (N_1,\dots,N_n) \in X, l_1,\dots,l_n \in \N\}
\]
is recognized by $\mathcal{A}$. We say $X$ is \textbf{$\alpha$-recognizable} if it is $\alpha$-recognized by some finite automaton.
\end{defi}

\noindent It follows easily from general facts about sets recognizable by finite automata that $\al$-recognizable sets are closed under boolean operations and coordinate projections (see \cite[Chapter 2.3]{automata}). A crucial tool is the following results from Hieronymi and Terry \cite{HT}.

\begin{thmC}[{\cite[Theorem B]{HT}}]\label{thm:HT}
  Let $\alpha$ be quadratic. Then
$\{ (x,y,z) \in \N^3 \ : \ x+y = z \}$ is $\alpha$-recognizable.
\end{thmC}

\noindent Next, recall $f_{\alpha}$ and $g_{\alpha}$ from~\eqref{eq:f} and Fact~\ref{fact:f}.

\begin{fact}\label{fact:alpharecorder} 
The sets \ts $A=\{ (u,v)\in \N^2 \, : \, u < v \}$ \ts and \ts  
$B=\{ (u,v)\in \N^2 \, : \, f_{\alpha}(u) < f_{\alpha}(v) \}$ \ts are $\alpha$-recognizable.
\end{fact}

\begin{proof}
For $A$, note that $u < v$ if and only if $\rho_{\alpha}(u)$ is lexicographically smaller than $\rho_{\alpha}(v)$ when read from right to left.
For $B$, let $u,v\in \N$. Take $b_1,b_1',b_2,b_2',\dots \in \N$ such that $\rho_{\alpha}(u) = b_{1}\; b_{2} \dots $ and $\rho_{\alpha}(v) = b'_{1}\; b'_{2} \dots $. Let $n$ be the smallest index where $b_{n} \neq b'_{n}$. It follows easily  from \eqref{eq:odd_even} and Fact \ref{fact:f} that when
$$
\gathered
n \text{ odd} \; : \; b_{n} < b'_{n} \quad \text{if and only if} \quad  f_{\al}(u) < f_{\al}(v),\\
n \text{ even} \; : \; b_{n} < b'_{n} \quad \text{if and only if} \quad  f_{\al}(u) > f_{\al}(v).
\endgathered
$$
(see~\cite[Fact~2.13]{H}). It is an easy exercise to construct a finite automaton that $\alpha$-recognizes $\{ (u,v)\in \N^2 \ : \ f_{\alpha}(u) < f_{\alpha}(v) \}$.
\end{proof}

\begin{lem}\label{lem:snap}
The graph of the function $g_{\alpha} : \N \to \N$ is $\alpha$-recognizable.
\end{lem}

\begin{proof}
We can assume that $\alpha$ is purely periodic, with minimum period $k$ (see Section~\ref{sec:periodic}).
By Fact~\ref{fact:nice_rec}, there are $\mu,\mu' \in \Q$ such that
\begin{equation*}
p_{n} = \mu q_{n} + \mu' q_{n+k} \quad \text{ for every } n \ge 0.
\end{equation*}
For $x \in \N$ with Ostrowski representation
$x = \sum_{n=0}^{N}b_{n+1}q_{n}$
we define:
$$\Shift(x) \;\coloneqq \; \sum_{n=0}^{N}b_{n+1}q_{n+k}.$$
In other words, if $\rho_{\alpha}(x) = b_{1}\; b_{2} \dots$, then $\rho_{\alpha}(\Shift(x)) = 0^{k}\; b_{1}\; b_{2} \dots$.
So $x \mapsto \Shift(x)$ is clearly $\alpha$-recognizable.
By Fact~\ref{fact:f}:
$$
g_{\alpha}(x) \;=\; \sum_{n=0}b_{n+1}p_{n} \;=\; \sum_{n=0} b_{n+1} (\mu q_{n} + \mu' q_{n+k}) \;=\; \mu \, x + \mu' \, \Shift(x).
$$
Since $g_{\alpha}(x)$ is a linear combination of $x$ and $\Shift(x)$, we see that $g_{\alpha}$ is $\alpha$-recognizable.
\end{proof}

\begin{prop}\label{prop:qe_free}
Let $F(\x)$ be a quantifier-free $\alpha$-\textup{PA} formula with variables $\x \in \Z^{d}$.
Then there is a NFA $\mathcal{A}$ of size $2^{\delta \, \ell(F)}$ that $\alpha$-recognizes $F$, in the sense that:
\begin{equation}\label{eq:NFA_qe_free}
\{ \x \in \Z^{d} \ : \ F(\x) = \text{true} \} = \{\x \in \Z^{d} \ : \ \ex \x' \in \Z^{d'} \  \mathcal{A} \text{ accepts } (\x,\x')\}.
\end{equation}
Here $\x' \in \Z^{d'}$ is a tuple of some auxiliary variables of length $d' \le \delta d$. $\mathcal{A}$ has the extra property that it accepts at most one tuple $(\x,\x') \in \Z^{d+d'}$ for every $\x \in \Z^{d}$.
Finally, the constant $\delta$ only depends on $\alpha$.
\end{prop}

\begin{proof}
    Each single variable $x$ in $F$ takes value over $\Z$, but can be replaced by $x_{1}-x_{2}$ for two variables  $x_{1},x_{2} \in \N$.
  Hence, we can assume that all variables take values over $\N$. It is easy to see that we can assume $F$ to be negation free.
 Recall that coefficients/constants in $\Z[\al]$ are given in the form $c\al + d$ with $c,d \in \Z$.
  Hence, each inequality in $F$ can be reorganized into the form:
  $$\a\,\y \; + \; \al\. \b\, \z \;\; \le \;\; \c\,\t  \; + \; \al\. \d\,\w. $$
Here $\a,\b,\c,\d$ are tuples of coefficients in $\N$, and $\y,\z,\t,\w$ are subtuples of $\x$.
For each homogeneous term $\a\,\y$, we use an additional variable $u = \a\,\y$ and replace each appearance of $\a\,\y$ in the inequalities by $u$.
Note that the length $\ell(F)$ grows at most linearly after adding all such extra variables.
The atoms in our formulas are either equalities of the form:
\begin{equation}\label{eq:star}
\tag{$\star$} \quad u = \a\,\y
\end{equation}
or inequalities of the form:
\begin{equation}\label{eq:dstar}
\tag{$\star\star$} \quad u + \al v \; \le \; w + \al z.
\end{equation}



\noindent We can rewrite each equality $u = \a\,\y$ into single additions by utilizing binary representations of the coefficients.
For example, the equality $u = 5y + 2z$ can be replaced by the following conjunction:
$$
y_{1} = y+y \; \land \; y_{2}=y_{1}+y_{1} \; \land \; y_{3}=y_{2}+y \; \land\; z_{1}=z+z \;\land\; u = y_{3} + z_1\ts.
$$
Note that the number of variables we introduce is linear in the length of the binary representation of $\a$. So we still $\ell(F)$ grows linearly when we introduce the new variables.
\smallskip

\noindent We have that $\alpha x = f_{\alpha}(x) + g_{\alpha}(x)$ for every $x \in \Z$.
Here $g_{\alpha}(x) \in \Z$ and $f_{\alpha}(x)$ always lies in the unit length interval $I_{\alpha}$.
For $u,v,w,z \in \NN$, we have $u+\al v < w + \al z$ if and only if:
\[
u + g_{\alpha}(v) < w + g_{\alpha}(z),\quad \text{or} \quad u + g_{\alpha}(v) = w + g_{\alpha}(z) \wedge  f_{\alpha}(v) < f_{\alpha}(z).\footnote{The case of a sharp inequality can be handled similarly.}
\]
Now we see that each atom in $F$, which is of type~\eqref{eq:star} or~\eqref{eq:dstar}, can be substituted by a Boolean combinations of simpler operation/functions, namely $f_{\al}$, $g_{\al}$, single additions $x+y=z$ and comparisons $x<y$.
We collect into a tuple $\x' \in \Z^{d'}$ all the auxiliary variables introduced in these substitutions. Observe that the total length $\ell(F)$, which also includes $d'$, only increased by some linear factor $\delta = \delta(\alpha)$.

\smallskip

\noindent
By Theorem~\ref{thm:HT}, Fact~\ref{fact:alpharecorder} and Lemma~\ref{lem:snap} each of the above simpler operations/functions can be recognized by a DFA of constant size. Using Fact~\ref{fact:automata_Boolean}, we can combine those DFA to get a DFA of size $2^{\gamma\, \ell(F)}$ that recognizes $F$ in the sense of~\eqref{eq:NFA_qe_free}.
Note that for each value $\x \in \Z^{d}$ of the original variables, the auxiliary $\x' \in \Z^{d'}$ are uniquely determined by $\x$.
\end{proof}


\begin{cor} Let $S\subseteq \N^n$ be $\al$-\textup{PA} definable. Then $S$ is $\al$-recognizable.
\end{cor}
\begin{proof}
Follows directly from the above proposition, combined with Fact~\ref{fact:automata_qe} for quantifier elimination.
\end{proof}

\begin{proof}[Proof of Theorem~\ref{th:quad-upper}]
  Consider an $\al$-PA sentence $S$ of the form~\eqref{eq:sentence}.
  Without loss of generality, we  can change domains from $\x_{i} \in {\Z^{n_{i}}}$ to $\x_{i}\in\N^{n_{i}}$, as shown in the proof of Proposition~\ref{prop:qe_free}.
  Also by negating $S$ if necessary,
  we can assume that $Q_{r} = \ex$.
  By Proposition~\ref{prop:qe_free}, there is an NFA $\mathcal A$ of size $2^{\delta\, \ell(F)}$ that $\alpha$-recognizes the quantifier-free part $\Phi(\x_{1},\dots,\x_{r})$ in the sense of~\eqref{eq:NFA_qe_free}.
  We can rewrite:
  \[
  \{ \x_{1} \. : \. Q_{2} \x_{2} \dots Q_{r} \x_{r} \; \Phi(\x_{1},\dots,\x_{r})\} = \{\x_{1}  \. : \. Q_{2} \x_{2} \dots Q_{r} \x_{r} \, \ex \x' \; \mathcal{A} \text{ accepts } (\x_{1},\dots,\x_{r},\x')\}.
  \]
  Since $Q_{r}=\ex$, we can group $\x_{r}$ and $\x'$ into one quantifier block.
  Repeatedly applying Fact~\ref{fact:automata_Boolean} and~\ref{fact:automata_qe} one after another, we can successively eliminate all $r-1$ quantifier blocks.
  This blows up the size of $\mathcal{A}$ by at most $r-1$ exponentiations.
  The resulting DFA $\mathcal{A}'$ has size at most a tower of height $r$ in $\delta\,\ell(F)$, and satisfies:
  $$
\{ \x_{1} \. : \. Q_{2} \x_{2} \dots Q_{r} \x_{r} \; \Phi(\x_{1},\dots,\x_{r})\} = \{\x_{1} \. : \. \mathcal{A'} \text{ accepts } \x_{1}\}
  $$
So deciding $S$ is equivalent deciding the whether ``$Q_{1}\x_{1} \; \mathcal{A'} \text{ accepts } \x_{1}$''. Note that $Q_{1}$ can be $\ex$ or $\for$.
However, since $\mathcal{A'}$ is deterministic, we can freely take its complement without blowing up its size. Thus, we can safely assume $Q_{1} = \ex$. Now, viewing $\mathcal{A'}$ as a directed graph, the sentence ``$\ex\x_{1} \; \mathcal{A'} \text{ accepts } \x_{1}$'' can be easily decided by a breadth first search argument. This can be done in time linear to the size of $\mathcal{A'}$.
\end{proof}


\bigskip

\section{Quadratic irrationals: $\PSPACE$-hardness}\label{sec:pspace}

We now give a proof of Theorem~\ref{th:quad-pspace}. Throughout this section, fix a $\pspace$-complete language  $\L \subseteq \{0,1\}^{*}$ and a $1$-tape Turing Machine (TM) $\M$ that decides $\L$. 
This means that given a finite input word $x \in \{0,1\}^{*}$ on its tape $\Tape$, the Turing machine $\M$ will run in space $\polyin(|x|)$ and output $1$ if $x \in \L$ and $0$ otherwise.

\medskip

The main technical theorem we establish in this section is the following:

\begin{thm}\label{th:pspace}
Let $\al\in \Qrc$ be a quadratic irrational. For every $s \in \N$, there is a map $X:\{0,1\}^{s} \to \N$ and an $\exists^{6}\forall^{4}\exists^{11}$  $\alpha$-\textup{PA} formula $\textup{\bf Accept}$ such that:
\[
 \textup{\bf Accept}(X(x)) \text{ holds if and only if } x \in \L.
\]
Both $X(x)$ and $\textup{\bf Accept}$ can be computed in time $O(s^c)$ for all $x \in \{0,1\}^s$ and all $s \in \N$, where the constant $c$ only depends on $\alpha$. Furthermore, the number of inequalities in $\textup{\bf Accept}$ only depends on $\alpha$ but not on $s$.
\end{thm}

%

\noindent The main argument of the proof of Theorem~\ref{th:pspace}  translates Turing machine computations into Ostrowski representations of natural numbers. This argument is presented in Subsection~\ref{sec:pspace_proof}. An explicit bound on the number of variables and inequalities for the constructed sentences are then given in Subsection~\ref{sec:analysis}, where we also treat the case $\al = \sqrt{2}$. Theorem \ref{th:quad-pspace} follows.

\medskip

\noindent
Before proving Theorem~\ref{th:pspace}, we construct in Subsection~\ref{sec:tools} some explicit $\alpha$-PA formulas to deal with the Ostrowski representation, exploiting the periodicity of the continued fraction expansion of $\al$.
Then we recall the definitions of Turing machine computations in Subsection~\ref{sec:pspace_basic}.

\subsection{Ostrowski representation for quadratic irrationals}\label{sec:tools}
Let $\al$ be a quadratic irrational. Recall from Section~\ref{sec:periodic} that we only need to consider a purely periodic $\alpha$ with minimum period $k$. Set $K := \lcm(2,k)$ and keep this $K$ for the remainder of this section.

\medskip

\noindent We first construct $\al$-PA formula that defines the set of convergents $(p_{n},q_{n})$ for which $K | n$.
Recall $\gamma_{i}$ from Remark~\ref{rem:gamma} (also see Fact~\ref{fact:div_k}).
Now define the $\al$-PA formula:
\begin{gather}
\ConvK(u,v,u',v') \; \coloneqq \; 1 < v < v' \; \land \; 0 < \al v - u \; \land \; \for w,z \; \label{eq:ConvK}\\
\bigwedge_{i=0}^{k+1} \Big( 0 < z < \gamma_{i+1}(v,v') \to |w - \al z| \ge |\gamma_{i}(u,u') - \al \gamma_{i}(v,v')| > |\gamma_{i+1}(u,u') - \al \gamma_{i+1}(v,v')| \Big). \nonumber
\end{gather}
\begin{lem} Let $u,v,u',v' \in \N$. Then
$\ConvK(u,v,u',v')$ holds if and only if $(u,v)=(p_{tK},q_{tK})$ and $(u',v')=(p_{tK+1},q_{tK+1})$ for some $t > 0$.
\end{lem}
\begin{proof}
First, the condition $\for w,z \; \big[ 0 < z < \gamma_{i+1}(v,v') \to \ldots \big]$ implies that the pairs $\big(\gamma_{i}(u,u'),\gamma_{i}(v,v')\big)_{0 \le i \le k+1}$ are $k+2$ consecutive convergents (see~\eqref{eq:conv2} and~\eqref{eq:consec}).
In other words, there is an $n > 0$ such that:
$$
\big(\gamma_{i}(u,u'),\gamma_{i}(v,v')\big) = (p_{n+i},q_{n+i}),\quad 0 \le i \le k+1.
$$
Also by Remark~\ref{rem:gamma}, we have $\big(\gamma_{0}(u,u'),\gamma_{0}(v,v')\big)=(u,v)$ and $\big(\gamma_{1}(u,u'),\gamma_{1}(v,v')\big)=(u',v')$.
So $(u,v)=(p_{n},q_{n})$ and $(u',v')=(p_{n+1},q_{n+1})$.
Then by~\eqref{eq:gamma_rec}:
$$\big(\gamma_{2}(u,u'),\gamma_{2}(v,v')\big)=
(a_{2}u'+u,a_{2}v'+v)=(a_{2}p_{n+1}+p_{n},a_{2}q_{n+1}+q_{n})$$
must be the next convergent $(p_{n+2},q_{n+2})$.
Combined with~\eqref{eq:rec}, we have
$$p_{n+2}=a_{n+2}p_{n+1}+p_{n}=a_{2}p_{n+1}+p_{n},$$
which implies $a_{n+2}=a_{2}$.
Similarly, we have $a_{n+i}=a_{i}$ for all $2 \le i \le k+1$.
Since $k$ is the minimum period of $\al$, we must have $k | n$.
Also because $0 < \al v - u = \al q_{n} - p_{n}$, we have $2 | n$ (see~\eqref{eq:odd_even}).
Therefore, $\ConvK(u,v,u',v') = \text{true}$ if and only if there is some $t\ge 1$ such that $(u,v)=(p_{t K},q_{t K})$ and $(u',v')=(p_{tK+1},q_{tK+1})$.
\end{proof}
\noindent In prenex normal form, $\ConvK$ is a $\for^{2}$-formula.

\medskip

\noindent
Next, we can also define the set of convergents $q_{n}$ for which $M | n$, where $M>10$ is some multiple of $K$ (see Subsection~\ref{ss:transcript} for why we need $M>10$).
To do this, we take a large enough prime $P$. There must exist some multiple $M$ of $K$ for which 
$$(q_{M},q_{M+1}) \equiv (q_{0},q_{1}) \mod{P}.$$
To see this, recall from Subsection~\ref{sec:periodic} that:
$$\begin{pmatrix}
    p_{mK+1} & p_{mK}\\
    q_{mK+1} & q_{mK}
  \end{pmatrix} \; = \;
\Gamma_{0} \, \cdots \, \Gamma_{mK+1} \; = \; \Gamma_{0} \, \Gamma_{1} \, (\Gamma_{2} \, \cdots \, \Gamma_{K-1} \, \Gamma_{0} \, \Gamma_{1})^m.
  $$
The matrix \. $\Gamma_{2} \cdots \Gamma_{K-1}\Gamma_{0}\Gamma_{1}$ \. is invertible mod~$P$ if $P$ is large enough. For every matrix $A$ invertible mod $P$, there is $m>0$ such that $A^m \equiv I \mod{P}$.
So there is a smallest $m>0$ such that:
$$
\begin{pmatrix}
    p_{mK+1} & p_{mK}\\
    q_{mK+1} & q_{mK}
\end{pmatrix}
\equiv
\Gamma_{0} \Gamma_{1}
=
\begin{pmatrix}
  p_{0} & p_{1} \\
  q_{0} & q_{1}
\end{pmatrix}
\. \mod{P}.
$$
Also by the recurrence~\eqref{eq:rec}, we have $(p_{mK+i},q_{mK+i}) \equiv (p_{i},q_{i}) \mod{P}$ for every~$i$.

\medskip

\noindent Clearly if $P$ is large enough then $M > 10$. Now define:
\begin{equation}\label{eq:ConvM}
\aligned
\ConvM(u,v,u',v') \; \coloneqq \; \ConvK(u,v,u',v') \;  \land \; v \equiv q_{0} \mod{P} \; \land \; v' \equiv q_{1} \mod{P}.
\endaligned
\end{equation}
Note that congruences can be expressed by $\for$-formula with one extra variable\footnote{We have $x_{1} \equiv x_{2} \mod P$ if and only if $\for w \; x_{1}-x_{2}-Pw=0 \; \lor \; |x_{1}-x_{2}-Pw| \ge P$.}.
So $\ConvM$ is a $\for^{3}$-formula in prenex normal form. To summarize:

\begin{lem} Let $u,v,u',v' \in \N$.
Then $\ConvM(u,v,u',v')$ holds if and only if there exists $t \ge 1$ such that $(u,v)=(p_{tM},q_{tM})$ and $(u',v')=(p_{tM+1},q_{tM+1})$.
\end{lem}

\noindent Recall from~\eqref{eq:Ost} that every $T \in \N$ has a unique Ostrowski representation:
\begin{equation*}
T = \sum_{n=0}^{N}b_{n+1}q_{n},
\end{equation*}
with $0 \le b_{1}<a_{1}$, $0 \le b_{n+1} \le a_{n+1}$ and $b_{n}=0$ if $b_{n+1}=a_{n+1}$.
We denoted $[q_{n}](T) \coloneqq b_{n+1}$. We often just write $[q_n]$ when the natural number $T$ is clear from the context.

\medskip


\noindent
In the proof of Theorem \ref{th:quad-pspace}, we consider numbers $T$ such that $[q_{n}](T) = 0$ when $n$ is odd, and $[q_n](T)$ is either $0$ or $1$ when $n$ is even. The reader can easily check that there is a bijection between the set of such natural numbers and finite words on $\{0,1\}$. We will use this observation through out this section.
In order to do so, we first observe that the above set of natural numbers is definable by the following $\alpha$-PA formula:
\begin{equation}\label{eq:Onlyk}
\gathered
\ZeroOne(T) \; \coloneqq \;  \for u,v,u',v' \; \Conv(u,v,u',v') \; \to \; \ex Z_{1},Z_{2},Z_{3} \\
        \Big( 0 > \al v - u \to \big[Z_{1} < v \; \land \; \After(u,v,u',v',Z_{2},Z_{3}) \; \land \; T = Z_{1} + Z_{2}\big] \Big) \; \land \\
\Big( 0 < \al v - u \to \big[ Z_{1} < 2v \; \land \; \Comp(u,v,u',v',Z_{1},Z_{2},Z_{3}) \; \land \; T = Z_{1} + Z_{2}\big] \Big)\ts,
  \endgathered
\end{equation}
where \hyperref[eq:After]{$\After$} and \hyperref[eq:Compatible]{$\Comp$} as defined in \eqref{eq:After} and \eqref{eq:Compatible}.
\begin{lem} Let $T\in \N$. Then $\ZeroOne(T)$ holds if and only if for all $n\in \N$
\begin{align*}
[q_{n}](T) &= 0 \when 2 \nmid n,\\
[q_{n}](T) &= 0,1 \when 2 | n.
\end{align*}
\end{lem}
\begin{proof}
The statement follows easily from \eqref{eq:odd_even} and Facts \ref{fact:After} and \ref{fact:Compatible}.
\end{proof}
\noindent Note that $\ZeroOne$ is a $\for^{4}\ex^{3}$-formula.

\medskip

\noindent Let $T,X$ be natural numbers. If $\ZeroOne(T)$ and $\ZeroOne(X)$, we can think of $T$ and $X$ as finite words on $\{0,1\}$. Thus, it is natural to ask whether we can express that  the word corresponding to $X$ is a prefix of the word corresponding $T$, by an $\alpha$-PA formula. It is not hard to see that the following $\alpha$-PA formula is able to do so:
\begin{equation}\label{eq:Pref}
\gathered
\Pref(X,T) \; \coloneqq \; \for u,v,u',v' \; \big( \Conv(u,v,u',v') \; \land \; v \le X \; \land \; X < v' \big) \to \\
\ex Z,Z' \; \Comp(u,v,u',v',X,Z,Z') \; \land \; T = X+Z.
\endgathered
\end{equation}
Note that $\Pref$  is an $\for^{4}\ex^{2}$-formula in prenex normal form.

\medskip

\subsection{Universal Turing machines}\label{sec:pspace_basic} Recall that we fixed a $\pspace$-complete language  $\L \subseteq \{0,1\}^{*}$ and a $1$-tape Turing Machine $\M$ that decides $\L$. Neary and Woods~\cite{NW} constructed a small universal 1-tape Turing machine (UTM) $$U = (\States, \Alph, \sigma_1, \delta, q_{1},q_{2}),$$ with $|\States| = 8$ states and $|\Alph| = 4$ tape symbols.\footnote{$\States$ -- states, $\Alph$ -- tape symbols, $\sigma_{1} \in \Alph$ -- blank symbol, $\delta : \States \times \Alph \to \States \times \Alph \times \{L,R\}$ -- transitions, $q_{1} \in \States$ -- start state, $q_{2} \in \States$ -- unique halt state.}

\medskip

\noindent Using $U$, we can simulate $\M$ in polynomial time and space. More precisely, let $x\in \{0,1\}^*$ be an input to $\M$. Then we can encode $\M$ and $x$ in polynomial time as a string $\encode{\M x} \in \Alph^{*}$.
Upon input $\encode{\M x}$, the UTM $U$ simulates $\M$ on $x$, and halts with one of the two possible configurations:
\begin{equation}\label{eq:yes_no}
\gathered
U(\encode{\M x}) = \text{``yes''} \when \M(x) = 1, \quad
U(\encode{\M x}) = \text{``no''}  \when \M(x) = 0.
\endgathered
\end{equation}
Here ``yes'' and ``no'' are the final state-tape configurations of $U$, which correspond to $\M$'s final configurations $(H, 10\dots)$ and $(H, 00\dots)$, respectively.
By the encoding in~\cite{NW}, these final ``yes''/``no'' configurations of $U$ have lengths $O(|\M|)$, which are constant when we fix~$\M$.
Furthermore, the computation $U(\encode{\M x})$ takes time/space polynomial in the time/space of the computation $\M(x)$.\footnote{It actually takes linear space and quadratic time.}
Since $\M(x)$ runs in space $\polyin(|x|)$, so does $U$ upon input $\encode{\M x}$. It is worth pointing out here that we take this detour via a universal
Turing machine to keep the number of variables constant.
\medskip

\noindent Let $\lambda = \lambda(|x|)$ be the polynomial bound on the tape length, so that the computation $U(\encode{\M x})$ always use less than $\lambda$ tape positions.

\medskip

\noindent Let $x \in \{0,1\}^{*}$. We now consider the simulation $U(\encode{\M x})$. Denote by $\Tape_{i}(x) \in \Alph^{\lambda-1}$ the contents of $U$'s tape on step $i$ of the computation $U(\encode{\M x})$. For $j\leq \lambda-1$, $\Tape_{i,j}(x) \in \Alph$ is the $j$-th symbol of $\Tape_{i}(x)$. Also denote by $\s_{i}(x) \in \States$ the state of $U$ on step $i$. We denote the $i$-th head position of $U$ by $\pi_i(x)$. Note that $1 \le \pi_{i}(x) \le \lambda-1$. As usual, we will suppress the dependence on $x$ if $x$ is clear from the context.

\medskip

\noindent Set $\B :=\{\markerblock\} \cup (\{\xsymb\} \times  \Alph) \cup (\States \times \Alph)$,  where $\xsymb$ is a special marker symbol. For each step $i$, we now encode the tape content $\Tape_{i}(x)$, the state $\s_{i}(x)$ and the tape head position $\pi_{i}(x)$ by the finite $\B$-word:
\begin{equation}\label{eq:T'i}
\Tape'_{i}(x) \; = \; [\xsymb,\xsymb] [\xsymb,\Tape_{i,1}]  \dots  [\xsymb,\Tape_{i,\pi_{i}-1}] \; [\s_{i},\Tape_{i,\pi_{i}}] \; [\xsymb,\Tape_{i,\pi_{i}+1}]  \dots  [\xsymb,\Tape_{i,\lambda-1}],
\end{equation}
 The marker block $\markerblock$ is at the beginning of each $\Tape'_{i}(x)$, which is distinct from the other $\lambda-1$ blocks in $\Tape'_{i}(x)$.
Note that $\Tape_{i}'(x)$ has in total $\lambda$ blocks. Observe $\Tape'_{1}(x)$ codes the starting configuration $\encode{\M x}$ of the simulation $U(\encode{\M x})$.
Now we concatenate $\Tape'_{i}$ over all steps $1 \le i \le \rho$, where $\rho$ is the terminating step of the simulation. We set
\[
\Transcript(x) := \Tape'_{1}(x) \; \dots \; \Tape'_{\rho}(x),
\]
and call $\Transcript(x)$ \emph{the transcript of $U$ on input  $\encode{\M x}$}, denoted by $\Transcript(x) = U(\encode{\M x})$. The last segment in $\Tape'_{\rho}(x)$ contains the ``yes'' configuration if and only if $\M(x)=1$. In total, $\Transcript(x)$ has $\lambda \rho$ blocks.

\medskip


\noindent Let $x\in \{0,1\}^{*}$. Let $\Transcript(x) = U(\encode{\M x})$.
Let $B_{t}(x) \in \B$ be the $t$-th block in $\T(x)$.
By the transition rules of $U$, the block $B_{t+\lambda}(x)$ depends on $B_{t-1}(x),B_{t}(x)$ and $B_{t+1}(x)$.
Thus, there is a function $f : \B^{3} \to \B$ such that for all inputs $x\in \{0,1\}^{*}$:
\begin{equation*}
B_{t+\lambda}(x) = f(B_{t-1}(x),B_{t}(x),B_{t+1}(x)) \quad\text{for every}\quad 0 \le t < \lambda(\rho-1).
\end{equation*}
Note that for the separator block $\markerblock$, we should have $f(B,\markerblock,B') = \markerblock$ for all $B,B'\in \B$.


\subsection{Proof of Theorem~\ref{th:pspace}}\label{sec:pspace_proof}
The detailed description of $X(x)$ and $\textup{\bf Accept}$ will be provided in Subsection~\ref{sec:final_construction}.
Here, we give an initial outline of the proof. We begin by associating to transcript $\Transcript$ a natural number $T$. Our goal then is to construct an $\alpha$-PA formula $\textbf{Accept}(X)$ consisting of four subformulas:
\begin{equation}\label{eq:pspace} 
\aligned
\ex T \;\; & \ZeroOne(T) \; \land \; \Pref(T,X)  \; \land \; \textbf{Transcript}_{\forall\exists}(T) \; \land  \; \End(T)
\endaligned
\end{equation}
In this formula, $\textbf{Transcript}_{\forall\exists}(T)$ ensures that there is a transcript $\Transcript$ to which $T$ corresponds. The formula $\Pref(X,T)$ guarantees that $X$ is a prefix of~$T$, and $\End(T)$ says that $T$ ends in ``yes''. We will need associate to each $x\in \{0,1\}^{*}$ an $X(x) \in \N$ such that the conclusion of Theorem~\ref{th:quad-pspace} holds.

\smallskip

\noindent For the rest of the proof, the meaning of $c_{i},d_{i},a,b$ will change depending on the context.
Recall the formulas \hyperref[eq:ConvK]{$\ConvK$}, \hyperref[eq:ConvM]{$\ConvM$}, \hyperref[eq:Onlyk]{$\ZeroOne$},  \hyperref[eq:Pref]{$\Pref$} from Section~\ref{sec:tools}.

\subsubsection{Encoding transcripts}\label{ss:transcript}
We first encode the transcripts $\Transcript$ by a number $T \in \N$ satisfying~$\ZeroOne(T)$.
Recall that $\Transcript$ is a finite $\B$-word, and observe that $|\B| = 37$. From now on, we view $\B$ as a set of $37$ distinct strings in $\{0,1\}^{6}$, each containing at least one~$1$.
Then we pick a large enough prime $P$ in \hyperref[eq:ConvM]{$\ConvM$} so that $M > 10$.
\begin{defi}\label{eq:match} Let $\Transcript$ be a transcript, and let $B_{t} \in \B$ be the $t$-th block in $\Transcript$. We associate to $\Transcript$ the natural number $T\in \N$ such that $\ZeroOne(T)$ and
for all $t\in \N$
\begin{enumerate}
\item $[q_{tM}](T) [q_{tM+2}](T)\dots [q_{tM+10}](T) = B_{t}$, and
\item $[q_{tM+12}](T) \dots [q_{(t+1)M-2}](T) = 0 \dots 0$.
\end{enumerate}
\end{defi}

\subsubsection{Constructing $\Next^{B,B',B''}$}
Let $B,B',B'' \in \B$, and let $u,v,u',v'\in \N$ and $t\ge 1$ be such that $(u,v)=(p_{tM},q_{tM})$ and $(u',v')=(p_{tM+1},q_{tM+1})$. We construct an $\al$-PA formula $\Next^{B,B',B'}(u,v,u',v',T)$ that holds if the block $B_{t+\lambda}$ in $T$ agrees with $f(B,B',B'')$.

\medskip

\noindent Let $r_{1} = \lambda M$ and $r_{2} = (\lambda+1)M$.
Then the block $B_{t+\lambda}$ of $\Transcript$ correspond to those $[q_{tM+i}](T)$ with $r_{1} \le i < r_{2}$.
By Fact~\ref{fact:div_k}, we can write each convergent $(p_{tM+i},q_{tM+i})$ with $r_{1} -1 \le i \le r_{2}$ as a linear combination $c_{i}(u,v) + d_{i}(u',v')$.
Note that the coefficients $c_{i},d_{i} \in \zz$ are independent of $t$, but do depend on $\lambda$.
They can be computed explicitly in time $\polyin(\lambda)$.
Let $\wt B = f(B,B',B'')$.
Then we sum up all $q_{tM+r_{1}+2j}$ for every $0 \le j < 6$ such that the $j$-th bit in $\wt B$ is `$1$'.
This sum can be expressed as $av + bv'$ for some $a,b \in \zz$ computable in time $\polyin(\lambda)$.
Again, $c_{i},d_{i}$ and $a,b$ depend on $\lambda$ and also the triple $B,B',B''$, but is independent of $t$.
Then $B_{t+\lambda} = \wt B$ if and only if we can uniquely write $T = W_{1} + (av + bv') + W_{2}$, where $W_{1} < q_{tM + r_{1} - 1}$ and $\Ost(W_{2}) \subset \{q_{n} : n \ge tM + r_{2}\}$.
Let $Z_{1} = W_{1} + (av + bv')$ and $Z_{2} = W_{2}$. Thus if $B_{t+\lambda} = \wt B$, they satisfy:
\begin{itemize}

\item[i)] $0 \le Z_{1} - (av + bv') <  q_{tM+r_{1}-1}\,$,

\item[ii)] $\Ost(Z_{2}) \; \subset \; \big\{q_{n} \,:\, n \ge tM + r_{2} \big\}\,$.

\end{itemize}
Both (i) and (ii) can be expressed using quantifier-free $\alpha$-PA formulas. For (i), recall that $q_{tM+r_{1}-1}$ is again linear combination of $v,v'$, so the corresponding $\alpha$-PA formula is just a conjunction of two linear inequalities in $Z_1,v,v'$. For (ii), using an auxiliary variable $Z_3$, we can express it as $\After(p_{tM+r_{2}-1}, q_{tM+r_{2}-1}, p_{tM+r_{2}}, q_{tM+r_{2}}, Z_{2}, Z_{3})$ (see~\eqref{eq:After}). Thus final $\al$-PA formula we want is:
\begin{equation}\label{eq:Next}
\Next^{B,B'B''}(u,v,u',v',T) \; \coloneqq  \; \ex Z_{1},Z_{2}, Z_{3} \;\;  \textup{i)}  \;  \land \; \textup{ii)} \; \land \; T = Z_{1} + Z_{2}.
\end{equation}

\subsubsection{Constructing $\Read^{B,B',B''}$}
Let $B,B',B'' \in \B$, and let $u,v,u',v'\in \N$ and $t\ge 1$ be such that $(u,v)=(p_{tM},q_{tM})$ and $(u',v')=(p_{tM+1},q_{tM+1})$. We will construct an $\al$-PA formula $\Read^{B,B',B''}(u,v,u',v',T)$ that holds if the three blocks $B_{t-1},B_{t},B_{t+1}$ in $T$ match with $B,B',B''$. Since the construction is very similar to the one of $\Next^{B,B',B''}$, we will leave verifying some of the details to the reader.
Note that the blocks $B_{t-1}B_{t}B_{t+1}$ in $T$ correspond to $[q_{n}](T)$ with $(t-1)M \le n < (t+2)M$.
So we just need to express $(p_{tM+i}, q_{tM+i})$ for $-M-1\le i \le 2M$ as linear combinations $c_{i}(u,v) + d_{i} (u',v')$.
Then we sum up all $q_{tM+i}$ that should correspond to the `$1$' bits in $B,B',B''$,
which is again some linear combination $av + bv'$.
This time the coefficients $c_{i},d_{i},a,b$ do \emph{not} depend on $\lambda$ and can be computed in \emph{constant} time.
Now we have $B_{t-1}B_{t}B_{t+1}=BB'B''$ if and only if we can uniquely write $T = Z_{1} + Z_{2}$, where $Z_{1}$ and $Z_{2}$ satisfy two conditions i'-ii') similar to i-ii) above. Again, we can express these two condition as quantifier free $\al$-PA formula. Thus the $\al$-PA formula we want is:
\begin{equation}\label{eq:Read}
\Read^{B,B',B''}(u,v,u',v',T) \; \coloneqq  \; \ex Z_{1},Z_{2},Z_{3} \;\; \text{i')} \; \land \; \text{ii')} \; \land \; T = Z_{1} + Z_{2}.
\end{equation}

\subsubsection{Recognizing transcripts}
A single transition of $T$ from $B,B',B''$ to $f(B,B',B'')$ can now be encoded in the $\al$-PA formula:
\begin{equation}\label{eq:Tran}
  \aligned
  \Tran^{B,B',B''}(u,v,u',v',T) \; \coloneqq  \; \Read^{B,B',B''}(&u,v,u',v',T) \\
  & \land \;  \Next^{B,B',B''}(u,v,u',v',T).
\endaligned
\end{equation}
Note that $\Tran$ is an $\ex^{6}$-formula.
Let $c,d\in \Q$ be such that $q_{(t+\lambda)M} = cq_{tM} + dq_{tM+1}$. To ensure that $T$ obeys the transition rule $f:\B^{3} \to \B$ everywhere, we simply require:
\begin{equation}\label{eq:patranscript}
\aligned
\textbf{Transcript}_{\forall\exists}(T) :=\\
 \for u,v,u',v' \; \big( \ConvM(u,v,u',v') \; \land \; cv + dv' \le T \big) &\to  \bigvee_{B, B', B'' \in \B} \Tran^{B,B',B''}(u,v,u',v',T).
\endaligned
\end{equation}
In this formula, $\ConvM(u,v,u',v')$ guarantees that  $v = q_{tM}$ is the beginning of some block $B_{t}$, and $q_{(t+\lambda)M} = cv + dv'$ is the beginning of the block $B_{t+\lambda}$, should it not exceed $T$. Thus for all $T\in \N$, we have $\textbf{Transcript}_{\forall\exists}(T)$ holds if and only if there is a transcript $\Transcript$ with $T(\Transcript)=T$.

\medskip

\noindent We now argue that  $\textbf{Transcript}_{\forall\exists}$ is a $\for^{4}\ex^{6}$ $\alpha$-PA formula. First, there are $\for^{4}$ variables $u,v,u',v'$.
Each $\Tran^{B,B',B''}$ is an $\ex^{6}$-formula, which also commutes with the big disjunction.
Also $\lnot \ConvM$ is an $\ex^{3}$-formula, which can be merged with the $\ex^{6}$ part.\footnote{We need to rewrite every implication ``$a \to b$'' as ``$\lnot a \lor b$''.}
Thus $\textbf{Transcript}_{\forall\exists}$ is a $\for^{4}\ex^{6}$ formula.

\medskip

\noindent
We need one last $\al$-PA formula to say that the computation corresponding to $T$ ends in the ``yes'' configuration (see~\eqref{eq:yes_no}).
Recall that ``yes'' has fixed length.
Assume ``yes'' starts at $v = q_{tM}$.
Then just like before, we can sum up all $q_{tM+i}$ that correspond to `$1$' bits in ``yes''.
This sum can be written as $a v + b v'$, with  $a,b \in \Z$ explicit \emph{constants} independent of $\lambda$.
Also observe that $q_{tM-1} = q_{tM+1} - a_{1}q_{tM} = v' - a_{1}v$. So we define an $\al$-PA formula as follows:
\begin{equation}\label{eq:End}
  \gathered
\End(T) \; \coloneqq \; \ex u,v,u',v',Z \; \;  \ConvM(u,v,u',v') \; \land \; Z < v' - a_{1}v \; \land \; T = Z + av + bv'.
\endgathered
\end{equation}
Observe that $\End(T)$ holds if and only if the computation corresponding to $T$ ends in ``yes''.
Note that $\End$ is a $\ex^{5}\for^{3}$-formula.

\subsubsection{Completing the construction}\label{sec:final_construction}
Finally, given $x \in \{0,1\}^{*}$, we can easily construct in time $\polyin(|x|)$ the content of the first segment $\Tape'_{1}(x)$ in $\Transcript(x)$ (see~\eqref{eq:T'i}).
Again, $\Tape'_{1}(x)$ is the starting configuration of the simulation $U(\encode{\M x})$, which is basically just $\encode{\M x}$. We denote by $X(x)$ the natural number $X$ such that $\ZeroOne(X)$
and for all $t\in \N$
\begin{enumerate}
\item $[q_t](X) =0$ for $t>10$,
\item $[q_{0}](X) [q_{2}](X) \dots [q_{10}](X) = \Tape'_{1}(x)$.
\end{enumerate}
It is easy to see that we can compute $X(x)$ in time $\polyin(|x|)$.

\medskip


\noindent Now construct the $\alpha$-PA formula $\textbf{Accept}(X)$:
\begin{equation}\label{eq:pspace1} 
\aligned
\ex T \;\; & \ZeroOne(T) \; \land \; \Pref(T,X)  \; \land \; \textbf{Transcript}_{\forall\exists}(T) \; \land  \; \End(T)
\endaligned
\end{equation}
From the construction it is clear that $\textbf{Accept}$ is an $\exists^{*}\forall^{*}\exists^{*}$ $\alpha$-PA formula such that:
\[
\textbf{Accept}(X(x)) \text{ holds if and only if } x\in \L.
\]
Finally, recall that in all constructed $\al$-PA formulas of Subsection~\ref{ss:transcript} and also $\textbf{Accept}$, the number of quantifiers/variables is constant, the number of linear inequalities (atoms) only depend on $\alpha$, and the linear coefficients/constants can be computed in time $\polyin(|x|)$. This completes the proof of Theorem~\ref{th:pspace}. \ $\sq$

\subsection{Analysis of $\textbf{Accept}$}\label{sec:analysis}
Recall that \hyperref[eq:Onlyk]{$\ZeroOne$}, \hyperref[eq:Pref]{$\Pref$} and $\textbf{Transcript}_{\forall \exists}$ in Section~\ref{sec:tools} are of the forms $\for^{4}\ex^{3}$ and $\for^{4}\ex^{2}$ and $\for^{4}\ex^{6}$ respectively.
Since we are taking their conjunctions, their outer $\for^{4}$ variables can be merged.
However, their $\ex$ variables need to be concatenated.
Overall, we have $\for^{4}\ex^{11}$ for  $\ZeroOne$, $\Pref$ and $\textbf{Transcript}_{\forall \exists}$.
The term $\End$ is $\ex^{5}\for^{3}$.
Merging its $\for^{3}$ variables with the other three terms, we have $\ex^{5}\for^{4}\ex^{11}$.
Lasly, we add in $\ex T$ and get a $\ex^{6}\for^{4}\ex^{11}$ sentence.

\medskip

\begin{figure}[h]
\label{table:var}
\begin{tabular}{ |c|c| }
 \hline
$x = y$ & $2$\\
  $|x| \ge |y|$, $|x| > |y|$ & $4$\\
$\After$, $\wt\After$ &  $4$\\
 $\Comp$ & $10$\\
$\Conv$ & $12$\\
$\ZeroOne$ & $34$\\
$\Pref$ & $26$\\
 $\Read$ & $8$\\
 $\Next$ & $8$\\
 $\Tran$ & $16$\\
 $\ConvK$ & $3 + 10(k+2)$ \\
 $\ConvM$ & $11 + 10(k+2)$ \\
 $\End$ & $14 + 10(k+2)$\\
 \hline
\end{tabular}
\caption{Number of inequalities of the various $\al$-PA formulas}
\label{fig:table}
\end{figure}

\noindent The number of inequalities in all constructed formulas is bounded in Figure \ref{fig:table}.
Overall, the number of inequalities in~\eqref{eq:pspace1} is at most:
$$
34 + 26 + 14 + 10(k+2) + 12 + 10(k+2) + 16\ts |\B|^3 \. = \. 810,534 + 20(k+2),
$$
where $k$ is the minimum period of the continued fraction of $\alpha$.  We conclude:

\begin{thm}
Deciding $\ex^{6}\for^{4}\ex^{11}$ $\alpha$-\textup{PA} sentences with at most \. $810,574 + 20\ts k_{\alpha}$ 
\ts inequalities is \ts $\pspace$-hard.
\end{thm}

\begin{cor}\label{cor:sqrt_2}
Deciding $\ex^{6}\for^{4}\ex^{11}$ $\sqrt{2}$-\textup{PA} sentences with at most $10^6$ inequalities is \ts $\pspace$-hard.
\end{cor}
\begin{proof}
Note that $\sqrt{2}+1 = [2; 2, \dots]$ has minimum period $k=1$.
\end{proof}

\bigskip

\section{Quadratic irrationals: General lower bound}\label{sec:lower_bound}

In this section, we establish Theorem~\ref{th:quad_lower}. Its proof follows the proof of  Theorem~\ref{th:pspace} very closely.
For a Turing machine $\mathcal{M}$ and $s \in \N$, recall that in Theorem \ref{th:quad-pspace} we constructed an $\forall^{6} \exists^{4}\forall^{11}$ $\al$-PA formula $\textbf{Accept}$ a function $X: \{0,1\}^{s} \to \N$ such that for every input $x\in\{0,1\}^{s}$, \ $\textbf{Accept}(X(x))$ holds if and only if $\mathcal{M}$ accepts $x$ within space $\polyin(s)$. Here we show that we can extend the space in which $\mathcal{M}$ accepts $x$ in exchange for adding alternating blocks of quantifiers in the $\al$-PA formula. For $\lambda\in \N$ we define
\[
g_{0}(s) = s \and g_{r+1}(s) = g_{r}(s)\, 2^{g_{r}(s)},\; r \ge 0.
\]
The following is the main theorem we establish in this section.
\begin{thm}\label{thm:lowerbound} Let $\al\in \Qrc$ be a quadratic irrational and $r \ge 1$. For every Turing machine $\mathcal{M}$ and $s \in \N$, there exist an $\al$-\textup{PA} formula $\textup{\bf Accept}$ with $(r+3)$ of alternating quantifier blocks and a function $X: \{0,1\}^{s} \to \N$ such that for every input $x\in\{0,1\}^{s}$,
\[
\textup{\bf Accept}(X(x)) \text{ holds if and only if $\mathcal{M}$ accepts $x$ within space $g_r(s)$.}
\]
Moreover, both $X(x)$ and $\textup{\bf Accept}$ can be computed in time $O(c\.s)$ for all $x \in \{0,1\}^{s}$ and all $s \in \N$, where the constant $c$ only depends on $\alpha$ and $M$.
\end{thm}
\noindent By a basic diagonalization argument the problem whether given a Turing machine $\M$ halts on an input string $x$ within space $g_{r}(|\M|+|x|)$ itself requires space at least $g_{r}(|\M|+|x|)$ to decide. Theorem \ref{th:quad_lower} follows.

\medskip

\noindent Recall that in \hyperref[eq:Next]{$\Next^{B,B'B''}$}, if $v=q_{tM}$ and $v'=q_{tM+1}$ then the shifted convergent $q_{(t+\lambda)M}$ can be written as $cv + dv'$, with  $c,d \in \zz$ having lengths $\polyin(\lambda)$.
  The resulting sentence~\eqref{eq:pspace1} has length $\polyin(\lambda)$, and is $\pspace$-complete to decide.
  To prove Theorem \ref{thm:lowerbound} we need to construct $\al$-PA formula $S_r$ such that $S_{r}$ has length $\polyin(\ell)$, at most $r-2$ alternating quantifiers, and defines the graph of the  shift map
\[
\textbf{Shift}_{r} \; : \; q_{tM} \mapsto q_{(t+g_r(\lambda))M}.
\]
   The following construction is classical.
  It was first used in Meyer~\cite{M} to prove that WS1S has non-elementary decision complexity, and was later improved on in Stockmeyer~\cite{S}.
  An expository version  is given  in Reinhardt~\cite{GTW}.
  For clarity and completeness, we reproduce it below in the setting of WS1S. 
  Afterwards, we translate it to $\al$-PA formulas.


\subsection{A lower bound for WS1S}
Let \textup{WS1S} be the weak monadic second order theory of $(\N,+1)$, that is the monadic second order logic of $(\N,+1)$ in which quantification over sets is restricted to quantification over finite subsets of $\N$. Formulas in the language of this theory are called \textup{WS1S}-formulas. We will use lower case letters $x,y,t,u,z$ to denote variables ranging over $\N$ and use upper case letters $A,C,D,E$ to denote variables ranging over finite subsets of $\N$.

\medskip



\noindent We think of each subset $X \in \mathcal{P}_{\textup{fin}}(\N)$ as an infinite word in $\{0,1\}$ that is eventually $0$. When we write $X = x_0x_1\dots x_n0^{\omega}$, we mean that $X$ is the finite set $\{ i \in \{0,\dots,n\} \ : \ x_i = 1\}$. The relation $i \in X$ simply means that the $i$-th digit $x_i$ is~1.

\begin{lemC}[\cite{GTW}]\label{lem:expon} Let $\lambda, r \in \N$. There exists a \textup{WS1S}-formula $F_{r+1}^{\lambda}(x,y,A,C)$ which holds if and only if:
$$
\aligned
y \; &= \; x + g_{r+1}(\lambda), \\
A \; &= \; 0^{x} | 1 0 0 \dots 0 | 1 0 0 \dots 0| 1 0 0 \dots 0 | 1 0 0 \dots 0| \dots |1 0 0 \dots 0| 1 0^{\omega}\., \\
C \; &= \; 0^{x} |0 0 0 \dots 0 | 1 0 0\dots 0| 0 1 0 \dots 0 | 1 1 0 \dots 0 | \dots |1 1 1 \dots 1| 0 0^{\omega}\.,
\endaligned
$$
where $A,C$ each has $2^{g_{r}(\lambda)}+2$ blocks, all of which except the first and last have length $g_{r}(\lambda)$.
\end{lemC}

\noindent In statement of Lemma \ref{lem:expon}, the separator $|$ is inserted just to improve the readability. The blocks in $C$ represent the integers $0,1,\dots,2^{g_{r}(\lambda)}-1$ in binary.
The blocks in $A$ mark the beginning of the blocks in $C$.
The first `1' in $A$ is at position $x$ and the last `1' in $A$ is at position $y$.\footnote{Position indexing starts at $0$.}
In total, the difference $y-x$ is $g_{r}(\lambda)2^{g_{r}(\lambda)} = g_{r+1}(\lambda)$.
\begin{proof}[Proof of Lemma \ref{lem:expon}]
We construct the $F^{\lambda}_r$ recursively, starting with the base case:
$$
F^{\lambda}_{0}(x,y,A,C) \; \coloneqq  \; y = x + \lambda.
$$
Here $x+\lambda$ represents $\lambda$ iterations of the successor function $s_{\N}$. For $F_0^\lambda$ we will not need any conditions on $A,C$.\newline

\noindent Let $r>0$ and suppose we have already constructed $F^{\lambda}_r(x,y,A,C)$ with the desired property. We will exploit the  fact the blocks in $C$ represent the integers $0,1,\dots,2^{g_{r}(\lambda)}-1$ in binary, and that adding $1$ to the integer represented by one block gives the integer represented by the next block. For that, we will use least-significant digit first encoding. We recall the carry rule for addition by $1$ in binary: if $X = x_{0} x_{1} \dots $, $Y = y_{0} y_{1} \dots$, then $\sum_{i=0}^{\infty} y_i 2^{-i} = 1 + \sum_{i=0} x_i 2^{-i}$ if and only if for all $i\in \N$
\begin{align*}
x_{0} & = \lnot y_{0}  \\
x_{i+1} &= \begin{cases}
\lnot  y_{i+1}, & \text{if $x_i=1$ and $y_i=0$,}\\
y_{i+1}, & \text{otherwise.}\\
\end{cases}
\end{align*}
The two conditions can be expressed by \textup{WS1S}-formulas:
\begin{gather}
0 \in X \; \leftrightarrow \; 0 \notin Y; \label{eq:inc}\\
(i \in X \land  i \notin Y) \; \leftrightarrow \; (i+1 \in X  \; \leftrightarrow \; i+1 \notin Y). \label{eq:carry}
\end{gather}
Observe that if we apply these rules on blocks of length $g_{r}(\lambda)$, starting with $0\dots0$, then we get:
$$0 0 \dots 0| 1 0 \dots 0 | 0 1 \dots 0| \dots |1 1 \dots 1| 0 0 \dots 0|1 0 \dots 0| \dots $$
So the blocks cycle back to $0 \dots 0$ eventually.
Thus we will characterize $C$ as the binary words obtained by  applying this transformation rule until the block $0 \dots 0$ is reached. We define:
\begin{align}
\label{largealign}
\begin{split}
F_{r+1}^{\lambda}(x,y,A,C)  \; \coloneqq  \;   x < y ;\;\; \for z,w,t,D,E \; \Big( \;  F_{r}^{\lambda}(z,w,D,E)   \;\to\; \\
 \Big[z = x \lor z = y \;\to\; z \in A,\, z \notin C;\;\; z < x \lor y < z  \;\to\; z \notin A,\, z \notin C;\\
   z = x,\, z < t < w  \;\to\; t \notin A,\, t \notin C;\;
   x \le z < w \le y \;\to\; (z \in A \leftrightarrow w \in A); \\
  z \in A,\, w < y \;\to\; (z \in C \leftrightarrow w \notin C); \\ 
  x \le z <  w < y,\, z+1 \notin A \;\to\; \big(z \in C,\, w \notin C \;\leftrightarrow\; (z+1 \in C \leftrightarrow w+1 \notin C) \big); \\
  x \le z<  w < y,\, z+1 \in A \;\to\; (z \in C \to w \in C); \\
 w = y,\, z \le t < w \;\to\; t \in C 
  \Big] \; \Big).
\end{split}
\end{align}
For readability, we use commas and semicolons to denote conjunctions of atoms and sub-clauses in~\eqref{largealign}.
Lines~2-3 of~\eqref{largealign} set up the block structures in $A$ and $C$. They make sure that $A$ and $C$ are empty outside the range $[x,y]$, and that the blocks in $A$ are of the form $1 0 0 \dots$.
Line 4  of \eqref{largealign} expresses the increment rule~\eqref{eq:inc} for every two consecutive blocks in $C$.
Here $z,w$ represent the first digits in two consecutive blocks.
Line~5 of of \eqref{largealign} expresses the carry rule~\eqref{eq:carry}.
Line~6  of \eqref{largealign} ensures that the blocks in $C$ do not cycle back to $0\dots0$,
because their last digits cannot decrease from~1 down to~0.
The last line of~\eqref{largealign} ensures that the last block in $C$ is $1\dots1$.
\end{proof}

\noindent By induction, it is easy to see that $F^{\lambda}_{r}$ has $r$ alternating quantifier blocks, starting with $\for$.
Observe that $F^{\lambda}_{r+1}$ has $5$ more variables than $F^{\lambda}_{r}$, namely $z,w,D,E,t$. Therefore, we have again by induction that $F_{r}^{\lambda}$ has at most $5(r+1)$ variables.
We can also bound their lengths:
$$
\ell(F^{\lambda}_{0}) = O(\lambda) \and
\ell(F^{\lambda}_{r+1}) = \ell(F^{\lambda}_{r}) + O(1) = O(\lambda + r).
$$
Here $\ell(F_{0}^{\lambda})=O(\lambda)$ instead of $O(1)$ because we needed to iterate the successor function $s_{\N}$ $\lambda$ times to represent $y=x+\lambda$.


 \begin{thm}\label{th:Buchi_lower_bound}
 Deciding \textup{WS1S}-sentences $S$ with $k+3$ alternating quantifiers in requires space at least:
 $$\rho\, 2^{ \, \iddots^{ \; 2^{ \, \eta \ell(S) } }} \ \text{, \ where  the tower has height $k$,}$$
 and $\rho$, $\eta$ are absolute constants.
 \end{thm}
\begin{proof}[Sketch of proof]
Consider the following decidable problem: Given a Turing machine $\M$ and an input string $X$, does $\M$ halt on $X$ within space $g_{r}(|\M|+|X|)$?
By the same construction as in Theorem~\ref{th:pspace}, we can write down a \textup{WS1S}-sentence $S$ with length $O(|\M|+|X|)$ so that $S$ holds if and only if $\M$ halts on $x$ within space $g_{r}(|\M|+|X|)$.
 Here $\lambda = \Omega(|\M|+|X|)$.
The last part $[\for u,v,u',v' \dots]$ in~\eqref{eq:pspace1} should be replaced by:
  $$
\for x,y,A,C \; F_{r}(x,y,A,C) \to \text{ transition rules} \dots
$$
Here $x$ and $y$ are bits in the transcript $\Transcript = U(\encode{\M X})$, with $y = x + g_{r}(\lambda)$.\footnote{Here $U$ is the universal TM used to emulate $\M(X)$.}
The resulting sentence $S$ has the form $\ex \dots \for \dots  \lnot F_{r} \lor \dots$
Since $F_{r}$ has $r$ alternating quantifiers, $S$ has $r+2$ alternating quantifiers.
The length $\ell(S)$ is roughly the input length $|\M|+|X|$ plus $\ell(F_{r})$, which is also $O(|\M|+|X|)$.
\end{proof}

\subsection{Proof of Theorem~\ref{thm:lowerbound}}
  We first translate the \textup{WS1S}-formula $F^{\lambda}_{r}(x,y,A,C)$ with $r$ alternating quantifiers into an $\al$-PA formula $S_{r}$ with $(r+1)$ alternating quantifiers.
  To do this, we replace in $F^{\lambda}_r$ each variable $x$ ranging over individuals by a separate quadruple $(u_{x},v_{x},u'_{x},v'_{x})$, where $(u_{x},v_{x}) = (p_{xM},q_{xM})$ and $(u'_{x},v'_{x})=(p_{xM+1},q_{xM+1})$,
and add the condition $\hyperref[eq:ConvM]{\ConvM}(u_{x},v_{x},u'_{x},v'_{x})$.
We replace each variable ranging over sets by an integer variable. We replace the relation $x \in X$ in $F^{\lambda}_r$ by  the relation whether $x$ is in $\Ost(X)$. By ~\eqref{eq:In_ex} and~\eqref{eq:In_for}, this relation is definable by an $\ex$ $\al$-PA formula and by an $\forall$ $\al$-PA formula. Recall from Fact~\ref{fact:div_k} that there  are constants $c,d\in \Z$ such that if $v = q_{tM}$ and $v' = q_{tM+1}$, then $q_{(t+1)M} = c v + d v'$ for all $t\in \N$.
We replace every $x+1$ term in $F_r^{\lambda}$ by $c v_{x} + d v'_{x}$.
Similarly, note there are  $c_{\lambda},d_{\lambda} \in \Z$ with  $\log(c_{\lambda}),\log(d_{\lambda}) = O(\lambda)$ such that $q_{(t+\lambda)M} = c_{\lambda} v + d_{\lambda} v'$ for all $t$.
So the relation $y = x+\lambda$ in $F_{0}^{\lambda}$ is replaced by $v_{y} = c_{\lambda}v_{x} + d_{\lambda}v'_{x}$.
Observe that $S_{0}$ has just $O(1)$ terms, instead of $O(\lambda)$ terms like $F_{0}$.
By induction, $S_{r}$ has $O(r)$ inequalities and variables.
The total length $\ell(S_{r})$ (including symbols and integer coefficients) is still $O(r+\lambda)$.

\medskip

\noindent Because of the $\ConvM$ predicate, $S_{0}$ now has one quantifier.
  For $r > 0$,  we can merge the $\for$ quantifiers in $\ConvM$ predicates with the $\for z,w,t,\dots$ quantifiers in $F_{r}^{\lambda}$ (of course replaced by quadruples).
  Because $x\in \Ost(X)$ is definable by both an $\ex$ PA-formula and an $\for$ $\al$-PA formula,
  the body of the sentence $S_{r+1}$, consisting of Boolean combinations in  $\in$/$\notin$, can be written using only $\for$ quantifiers.
  These extra $\for$ quantifiers can again be merged into the $\for z,w,t$ part.
  This means $S_{r+1}$ has only one more alternating quantifiers than $S_{r}$.
  So $S_{r}(u_{x},v_{x},u'_{x},v'_{x},u_{y},v_{y},u'_{y},v'_{y},A,C)$ is an $\al$-PA formula quantifier formula with $(r+1)$ alternating quantifier blocks.

\medskip

\noindent Now we are back to encoding Turing machine computations. We give a brief outline how the construction follows the proof of Theorem~\ref{th:pspace}. In the definition of $\textbf{Transcript}_{\forall\exists}$ in~\eqref{eq:patranscript}, we replace  $[\for u,v,u',v' \dots]$ by:
  $$\aligned
\for \, & u_{x},v_{x},u'_{x},v'_{x},u_{y},v_{y},u'_{y},v'_{y},A,C \;\; \\ 
& \qquad \quad \big( S_{r}(u_{x},v_{x},u'_{x},v'_{x},u_{y},v_{y},u'_{y},v'_{y},A, C)\;\land \; v_{y} \le \tau \big) 
\ \to \ \text{ transition rules}\dots
\endaligned
  $$
In these transition rules, \hyperref[eq:Read]{$\Read^{B,B',B''}$} is kept as before with $u_{x},v_{x},u'_{x},v'_{x}$, but \hyperref[eq:Next]{$\Next^{B,B',B''}$} can be rewritten using the shifted convergents $u_{y},v_{y}, u'_{y},v'_{y}$.
Altogether, this expresses the transition rule for each jump $y = x + g_{r}(\lambda)$.
The resulting formula $\textbf{Transcript}'$ has the form $\for \dots  \lnot S_{r} \lor \dots$. Since $S_{r}$ has $r+1$ alternating quantifiers, this formula has $r+2$ alternating quantifiers.
We now construct $\textbf{Accept}_{\mathcal{M},s,r}$ using $\textbf{Transcript}'$ as is in \eqref{eq:pspace1}. This $\al$-PA formula has $r+3$ alternating blocks of quantifiers and the number of variables and inequalities used is just  $O(r)$.

\bigskip

\section{Non-quadratic irrationals: Undecidablity}
In this section, we consider the case that $\al$ is non-quadratic. As pointed out in the introduction it follows from \cite{HT-Proj} that $\al$-PA is undecidable whenever $\al$ is non-quadratic. Here we will show that even $\al$-PA sentences with only four alternating quantifier blocks are undecidable. We prove a slightly stronger result for which we have to introduce an extension of $\al$-PA.

\medskip

\noindent Let $K$ be a subfield of $\R$. An \emph{$K$-Presburger sentence} (short: \emph{$K$-PA sentence}) is a statement of the form
\begin{equation}\label{eq:sentenceab}
Q_{1} \x_{1} \in \zz^{n_{1}} \; \dots \; Q_{r} \x_{r} \in \zz^{n_{r}} \;\; \Phi(\x_{1},\dots,\x_{r}),
\end{equation}
where $Q_{1},\dots,Q_{r} \in \{\for,\ex\}$ are $r$ alternating quantifiers, and $\Phi$ is a Boolean combination of linear inequalities in $\x_{1},\dots,\x_{r}$
with coefficients and constant terms in $K$. We define $K$-PA formulas and other relevant notations analogous to the case of $\alpha$-PA sentences in Section \ref{section:alphapa}. The following is the main result of this section.

\begin{thm}\label{thm:undtwo} Suppose that $1,\al,\beta$ are $\Q$-linearly independent. Then $\exists^k \forall^k\exists^k \forall^k$ $\Q(\al,\beta)$-PA sentences are undecidable, where $k=20000$.
\end{thm}

\noindent When $\alpha$ is non-quadratic, then $1,\alpha,\alpha^2$ are $\Q$-linearly independent. As $\Q(\al,\al^2)=\Q(\al)$, Theorem \ref{th:nonquad-undec} follows from Theorem \ref{thm:undtwo}.

\subsection{Further tools}
In this section we are working with two different irrationals $\alpha$ and $\beta$, and we will need to refer to the Ostrowski representation based on $\alpha$ and $\beta$.
We denote by $p_{n}/q_{n}$ and $p'_{n}/q'_{n}$ the $n$-th convergent of $\alpha$ and  $\beta$, respectively.
Let $\Ost_{\alpha} \coloneqq  \{ q_{n} \, : \, n\in \N\}$ and  $\Ost_{\beta} \coloneqq  \{ q'_{n} \, : \, n\in \N\}$.
For $X \in \N$, denote by $\Ost_{\alpha}(X)$ the set of $q_{n}$ with non-zero coefficients in the $\alpha$-Ostrowski representation of $X$.
Then $\Ost_{\beta}(X)$ is defined accordingly for the $\beta$-Ostrowski representation of $X$.
All earlier notations can be easily adapted to $\alpha$ and $\beta$ separately.
For brevity, we define the remaining functions and notations just for $\alpha$.
The corresponding versions for $\beta$ are defined accordingly, with obvious relabelings.

\medskip


\nin
For $X \in \N$ with $\alpha$-Ostrowski representation $X = \sum_{n=0}^{\infty} b_{n+1} q_{n}$ and $d\in \Ost_{\alpha}$, define
\begin{equation}\label{eq:restr}
\restr{X}{d}{\alpha} \; \coloneqq \;  \sum_{n\in\N,\, q_{n} \leq d} b_{n+1} q_{n}.
\end{equation}
The relation $Y = \restr{X}{d}{\alpha}$ from~\eqref{eq:restr} is $\exists^{2}$-definable by $\al$-PA formula:
\[
 Y < v^{+} \; \land \; \ex Z,Z' \; \Comp(u,v,u^{+},v^{+},Y,Z,Z') \; \land \; Y + Z = X.
\]
Here $\Comp$ is from~\eqref{eq:Compatible}.

\medskip



\noindent The function $f_{\alpha}$ defined in \eqref{eq:f} and its interaction with the corresponding function $f_{\beta}$ play a crucial role. We collect two easy facts about $f_{\alpha}$ here.

\begin{fact}\label{fact:ostlocal} Let $X\in \N$. Then there is an interval $I\subseteq \R$ around $f_{\alpha}(X)$ and $d\in \Ost_{\alpha}$ such that for all $Y\in \N$
\[
f_{\alpha}(Y) \in I \Longrightarrow \restr{Y}{d}{\alpha} = X.
\]
\end{fact}
\begin{proof}
Let $\sum_{n=0}^{m} b_{n+1} q_{n}$ be the $\alpha$-Ostrowski representation of $X$.
Without loss of generality, we may assume that $\alpha q_{m} - p_{m} >0$.
Then set
\[
Z_2 = X + q_{m+2} \hbox{ and } Z_1 = X + q_{m+3}.
\]
Since $\alpha q_{m+2} - p_{m+2} >0$ and $\alpha q_{m+3} - p_{m+3} <0$, we get from Fact \ref{fact:f} that
\[
f_{\alpha}(Z_1) < f_{\alpha}(X) < f_{\alpha}(Z_2).
\]
Now it follows easily from \cite[Fact 2.13]{H} and Fact \ref{fact:f} that for all $Y\in \N$
\[
f_{\alpha}(Z_1) < f_{\alpha}(Y) < f_{\alpha}(Z_2) \Longrightarrow \restr{Y}{q_{m}}{\alpha} = X,
\]
as desired.
\end{proof}

\begin{fact}\label{fact:ostlocal2} Let $X\in \N$ and let $J\subseteq \R$ be an open interval around
$f_{\alpha}(X)$. Then there is $d\in \Ost_{\alpha}$ such that for all $Y\in \N$
\[
 \restr{Y}{d}{\alpha} = X \Longrightarrow f_{\alpha}(Y) \in J.
\]
\end{fact}
\begin{proof}
Let $\sum_{n=0}^{m} b_{n+1} q_{n}$ be the $\alpha$-Ostrowski representation of $X$. Let $n\in \N$ be such that
\begin{itemize}
\item $n > m+1$,
\item $\alpha q_{n} - p_{n}>0$ and
\item $\big(f_{\alpha}(X) + (\alpha q_{n+1} - p_{n+1}),f_{\alpha}(X) + (\alpha q_{n} - p_{n})\big) \subseteq J$.
\end{itemize}
Let $Y \in \N$ be such that $\restr{Y}{q_{n+2}}{\alpha} = X$. It is left to show that $f_{\alpha}(Y) \in J$. By Fact \ref{fact:f} and \cite[Fact 2.13]{H} we get that
\[
f_{\alpha}(X) + (\alpha q_{n+1} - p_{n+1}) = f_{\alpha}(X+q_{n+1})<f_{\alpha}(Y) < f_{\alpha}(X + q_{n}) = f_{\alpha}(X) + (\alpha q_{n} - p_{n}).
\]
Thus $f_{\alpha}(Y)\in J$.
\end{proof}

\subsection{Uniform definition of all finite subsets of $\N^2$} Let $\alpha,\beta$ be two positive irrational numbers such that $1,\alpha,\beta$ are $\Q$-linearly independent. The goal of this section is to produce a $6$-ary $\Q(\alpha,\beta)$-PA formula $\Member$ such that for every set $S\subseteq \N^2$ there is $\bm{X} \in \N^4$ such that for all $(s,t) \in \N^2$,
\[
(s,t) \in S \Longleftrightarrow \Member(\bm{X},s,t).
\]
The $\Q$-linear independence of $1,\alpha,\beta$ is necessary as we will see that the existence of such a relation implies the undecidability of the theory. The failure of our argument in the case of $\Q$-linear dependence of $1,\alpha,\beta$ can be traced back to the fact that the following lemma fails when $1,\alpha,\beta$ are $\Q$-linearly dependent.

\medskip

\nin
Hereafter, we let $\cj X = (X_{1},X_{2}), \cj Y = (Y_{1},Y_{2})$ and $\cj Z = (Z_{1},Z_{2})$.

\begin{lem}\label{lem:injectivity}
Let $\cj X, \cj Y \in \N^2$. Then
\[
|f_{\alpha}(X_1) - f_{\beta}(X_2)|=|f_{\alpha}(Y_1) - f_{\beta}(Y_2)| \;\Longrightarrow\; \cj X= \cj Y.
\]
\end{lem}
\begin{proof}
Assume the LHS holds, then there are $U_1,U_2,V_1,V_2\in \N$ such that:
\[
\bigl|\alpha X_1 - U_1 - \beta X_2 + U_2\bigr|  \, = \, \bigl|\alpha Y_1 - V_1 - \beta Y_2 + V_2\bigr|\ts.
\]
By $\Q$-linear independence of $1,\alpha,\beta$, we get that $X_1 = Y_1$ and $X_2 = Y_2$.
\end{proof}

\begin{defi} Define $g : \N^4\to \R$ to be the function that maps $(\cj X, \cj Y)$ to
\[
\big|f_{\alpha}(X_2) - f_{\alpha}(X_1)-|f_{\alpha}(Y_2) - f_{\beta}(Y_1)|\big|.
\]
\end{defi}

\begin{defi}
  Let $\Best$ be the relation on $\N \times \N \times \N^2\times \N$ that holds precisely for all tuples $(d,e,\cj X,Y_{1})$ for which there is a $Y_{2} \in \N$ such that
\begin{itemize}
\item $Y_{1} \leq d$, $Y_{2} \leq e$,
\item $g(\cj X, \cj Y) < g(\cj X,\cj Z)$ for all $\cj Z \in \N_{\leq d} \times \N_{\leq e}$ with $\cj Z \neq \cj Y$.
\end{itemize}
\end{defi}

\noindent Observe that for given $(d,e,\cj X) \in \N \times \N \times \N^2$ there is at most one $Y_{1} \in \N_{\leq d}$ such that $\Best(d,e,\cj X,Y_{1})$ holds. We will later see in Lemma \ref{lem:interior} that for given $d \in \N$ we can take $e\in \N$ large enough such that for all $X_1 \in \N$ and  $Y_{1} \leq d$ the set
\[
\{ X_2\in \N \ : \ \Best(d,e,X_1,X_2,Y)\}
\]
is cofinal in $\N$.

\begin{lem}\label{lem:Best} $\Best $ is definable by an $\exists^5 \forall^4$ $\Q(\alpha,\beta)$-PA formula.
\end{lem}
\begin{proof}
Observe that $\Best(d,e,\cj X,Y_1)$ holds if and only if
\begin{equation*}
\aligned
\ex & Y_2, U_1, U_2, V_1, V_2  \quad \for Z_1, Z_2, W_1, W_2 \quad \bigg[ Y_1 \leq d \wedge Y_2 \leq e \wedge \\
& f_{\alpha}(X_1)=\alpha X_1 - U_1  \wedge f_{\alpha}(X_2)=\alpha X_2 - U_2 \wedge f_{\alpha}(Y_1)=\alpha Y_1 - V_1 \wedge f_{\beta}(Y_2)=\beta Y_2 - V_2 \wedge \\
& \Big[ \big( Z_1 \leq d \wedge Z_2 \leq e \wedge f_{\alpha}(Z_1) = \alpha Z_1 - W_1 \wedge f_{\beta}(Z_2)=\beta Z_2 - W_2\\
& \wedge (Z_1,Z_2)\neq (Y_1,Y_2) \big) \rightarrow \big|(\alpha X_2 - U_2) - (\alpha X_1 -U_1)-| (\beta Y_2 - V_2) - (\alpha Y_1 - V_1)|\big| \\
& \hspace{11.5em} < \big|(\alpha X_2 - U_2) - (\alpha X_1 -U_1)-| (\beta Z_2 - W_2) - (\alpha Z_1 - W_1)|\big| \Big]\bigg].
\endaligned
\end{equation*}
This implies the result.
\end{proof}

\noindent The following lemma is crucial in what follows. It essentially says that for every subinterval of $I_{\alpha}\cap I_{\beta}$ and every $d\in \Ost_{\alpha}$, we can recover $(\Ost_{\alpha})_{\leq d}$ just using parameters from this interval and $\Ost_{\beta}$. This should be compared to condition (ii) in \cite[Th.~A]{HT-Proj}.

\begin{lem}\label{lem:interior} Let $d\in \Ost_{\alpha}$, $e_0\in \Ost_{\beta}$, $\cj X \in \N^{2}$ and $s\in \N$ be such that
\begin{enumerate}
\item $f_{\alpha}(X_1), f_{\alpha}(X_2) \in I_{\beta}$,
\item $f_{\alpha}(X_1) < f_{\alpha}(X_2)$,
\item $s\leq d$.
\end{enumerate}
Then there is $e\in \Ost_{\beta}$ and an open interval $J\subseteq \big(f_{\alpha}(X_1),f_{\alpha}(X_2)\big)$ such that $e\geq e_0$ and for all $Z\in \N$
\[
f_{\alpha}(Z) \in J \Longrightarrow \Best (d,e,X_1,Z,s).
\]
\end{lem}
\begin{proof}
Let $e \in \Ost_{\beta}$ be large enough such that for every $w_1\in \N_{\leq d}$ there is $w_2 \in \N_{\leq e}$ such that
\[
f_{\alpha}(w_1) \in I_{\beta} \Longrightarrow |f_{\alpha}(w_1) - f_{\beta}(w_2)| < f_{\alpha}(X_2)-f_{\alpha}(X_1).
\]
The existence of such an $e$ follows from the finiteness of $\N_{\leq d}$ and the density of $f_{\beta}(\N)$ in $I_{\beta}$. Let $w\in \N_{\leq e}$ be such that
\[
|f_{\alpha}(s) - f_{\beta}(w)|  < f_{\alpha}(X_2)-f_{\alpha}(X_1).
\]
By Lemma \ref{lem:injectivity} we can find an $\varepsilon>0$ such that for all $(w_1,w_2)\in \N_{\leq d} \times \N_{\leq e}$ with $(w_1,w_2)\neq (s,w)$
\[
\big||f_{\alpha}(w_1) - f_{\beta}(w_2)|-|f_{\alpha}(s) - f_{\beta}(w)|\big| > \varepsilon.
\]
Set
\[
\delta \coloneqq  f_{\alpha}(X_1) + |f_{\alpha}(s) - f_{\beta}(w)|.
\]
Set $J \coloneqq  (\delta - \frac{\varepsilon}{2}, \delta + \frac{\varepsilon}{2})$. Let $Z\in \N$ be such that $f_{\alpha}(Z) \in J$. It is left to show that $\Best(d,e,X_1,Z,s)$ holds. We have that for all $(w_1,w_2)\in \N_{\leq d} \times \N_{\leq e}$ with $(w_1,w_2)\neq (s,w)$
\begin{align*}
g(X_1,Z,w_1,w_2) &= \big| f_{\alpha}(Z) - f_{\alpha}(X_1) - |f_{\alpha}(w_1) - f_{\beta}(w_2)|\big| \\
&= \big| f_{\alpha}(Z) - \delta + |f_{\alpha}(s) - f_{\beta}(w)| - |f_{\alpha}(w_1) - f_{\beta}(w_2)|\big|\\
&\geq \Big| |f_{\alpha}(Z) - \delta| - \big||f_{\alpha}(s) - f_{\beta}(w)| - |f_{\alpha}(w_1) - f_{\beta}(w_2)|\big|\Big| > \frac{\varepsilon}{2}.
\end{align*}
Moreover,
$$
g(X_1,Z,s,w) \, = \, \big| f_{\alpha}(Z) - f_{\alpha}(X_1) - |f_{\alpha}(s) - f_{\beta}(w)| \big|
\, \leq \, \big| f_{\alpha}(Z) - \delta \big|\leq \frac{\varepsilon}{2}\ts.
$$
Thus $\Best(d,e,X_1,Z,s)$ holds, as desired.
\end{proof}

\begin{lem}\label{lem:density} Let $d\in \Ost_{\alpha},s\in \N, \cj X \in \N^2$ be such that
\begin{enumerate}
\item $f_{\alpha}(X_1), f_{\alpha}(X_2) \in I_{\beta}$,
\item $f_{\alpha}(X_1) < f_{\alpha}(X_2)$,
\item $s \leq d$.
\end{enumerate}
Then there are $e_1 \in \Ost_{\beta},e_2\in \Ost_{\alpha},Y\in \N$ such that
\begin{enumerate}[align=left]
\item[(i)] $f_{\alpha}(X_1) < f_{\alpha}(Y) < f_{\alpha}(X_2)$,
\item[(ii)] $d < e_1 < e_2$
\item[(iii)] for all $Z\in \N$
\[
\restr{Z}{e_2}{\alpha}=Y \Longrightarrow \Best (d,e_1,X_1,Z,s).
\]
\end{enumerate}
\end{lem}
\begin{proof}
By Lemma \ref{lem:interior} there is an open interval $J\subseteq \big(f_{\alpha}(X_1),f_{\alpha}(X_2)\big)$ and $e_1 \in \Ost_{\beta}$ such that $e_1>d$ and for all $Z\in \N$
\[
f_{\alpha}(Z) \in J \Longrightarrow \Best (d,e_1,X_1,Z,s).
\]
Take $Y\in \N$ such that $f_{\alpha}(Y) \in J$. By Fact \ref{fact:ostlocal2} we can find $e_2\in \Ost_{\alpha}$ arbitrarily large such that $f_{\alpha}(Z) \in J$ for all $Z\in \N$ with $\restr{Z}{e_2}{\alpha}=Y$. The statement of the Lemma follows.
\end{proof}

\begin{defi}\label{def:Adm_Mem}
  Define $\Adm$ to be the relation on $\Ost_{\alpha}^4 \times \Ost_{\beta}^2 \times \N^6$ that holds precisely for all tuples
\[
(d_1,d_2,d_3,d_4,e_1,e_2,X_1,X_2,X_3,X_4,s,t) \in  \Ost_{\alpha}^4 \times \Ost_{\beta}^2 \times \N^6
\]
such that
\begin{itemize}
\item $d_1, d_2,d_3$ are consecutive elements of $\Ost_{\alpha}(X_1)$,
\item $d_4\in \Ost_{\alpha}(X_3)$ and $d_1\leq d_4 < d_2$,
\item $e_1, e_2 \in \Ost_{\beta}(X_2)$ and $d_1 \leq e_1 < d_2 \leq e_2 < d_3$
\item $\Best (d_1,e_1,\restr{X_4}{d_1}{\alpha},X_4,s)$
\item $\Best (d_2,e_2,\restr{X_4}{d_2}{\alpha},X_4,t)$
\end{itemize}

\noindent Define $\Member$ to be the $6$-ary relation on $\N$ that exactly for all tuples $(X_1,X_2,X_3,X_4,s,t)\in \N^6$ such that there exist $d_1,d_2,d_3,d_4 \in \Ost_{\alpha},\, e_1,e_2\in \Ost_{\beta}$ with
\[
\Adm(d_1,d_2,d_3,d_4,e_1,e_2,X_1,X_2,X_3,X_4,s,t).
\]
\end{defi}

\begin{thm}\label{thm:codepairs} Let $S \subseteq \N^2$ be finite. Then there are $X_1,X_2,X_3,X_4\in \N$ such that for all $s,t\in \N$
\[
(s,t) \in S \Leftrightarrow \Member(X_1,X_2,X_3,X_4,s,t).
\]
\end{thm}
\begin{proof}
Let $S \subseteq \N^2$ be finite. Let $c_1,\dots,c_{2n} \in \N$ be such that
\[
S = \{(c_1,c_2),\dots,(c_{2n-1},c_{2n})\}.
\]
Recall that the convergents of $\alpha$ and $\beta$ are $\{p_{n}/q_{n}\}$ and $\{p'_{n}/q'_{n}\}$, respectively.
We will construct two strictly increasing sequences $(k_i)_{i=0,\dots,2n}, (l_i)_{i=1,\dots,2n}$ of non-consecutive natural numbers and another sequence $(W_{i})_{i=0,...,2n}$ of natural numbers such that for all $i=0,\dots,2n$
\begin{enumerate}
\item $W_{j}=\restr{W_{i}}{q_{k_j}}{\alpha}$ for all $j\leq i$, and $f_{\alpha}(W_i) \in I_{\beta}$,
\item $q_{k_i} > \max \{c_1,\dots, c_{2n}\}$,
\end{enumerate}
and furthermore if $i\geq 1$, then
\begin{enumerate}
\item[(3)] $q_{k_{i-1}} < q'_{l_i} < q_{k_{i}}$,
\item[(4)] for all $Z \in \N$
\[
\restr{Z}{q_{k_i}}{\alpha}=W_{i} \Longrightarrow \Best (q_{k_{i-1}},q'_{l_{i}},W_{i-1},Z,c_i).
\]
\end{enumerate}

\noindent We construct these sequences recursively. For $i=0$, pick $k_0\in \N$ such that
\[
q_{k_0} > \max \{c_1,\dots, c_{2n}\}.
\]
Pick $W_0\in \N$ such that $W_0=\restr{W_0}{q_{k_0}}{\alpha}$ and $f_{\alpha}(W_0) \in I_{\beta}$. Now suppose that $i>0$ and that we already constructed $k_0,k_1,\dots,k_{i-1}$, $l_1,\dots,l_{i-1}$ and $W_1,\dots, W_{i-1}$ such that the above conditions (1)-(4) hold for $j=1,\dots,i-1$. We now have to find $k_i,l_i$ and $W_i$ that (1)-(4) also hold for $i$. We do so by applying Lemma \ref{lem:density}. By Fact \ref{fact:ostlocal} we can take $T\in \N$ such that
\begin{itemize}[align=left]
\item[(a)] $f_{\alpha}(T) > f_{\alpha}(W_{i-1})$, $\restr{T}{q_{k_{i-1}}}{\alpha}= W_{i-1}$, $f_{\alpha}(T) \in I_{\beta}$ and
\item[(b)] for all $Z\in \N$, $\big(f_{\alpha}(W_{i-1}) < f_{\alpha}(Z) < f_{\alpha}(T) \Longrightarrow \restr{Z}{q_{k_{i-1}}}{\alpha}=W_{i-1}\big)$.
\end{itemize}
We now apply  Lemma \ref{lem:density} with $X_1\coloneqq W_{i-1}, X_2\coloneqq  T, d \coloneqq  q_{k_{i-1}}$ and $s \coloneqq  c_{i-1}$.
We obtain $e_1\in \Ost_{\beta}$, $e_2 \in \Ost_{\alpha}$ and $Y \in \N$ such that $d< e_1<e_2$, $f_{\alpha}(W_{i-1}) < f_{\alpha}(Y) < f_{\alpha}(T)$ and for all $Z\in \N$
\[
\restr{Z}{e_2}{\alpha}=Y \Longrightarrow \Best(q_{k_{i-1}},e_1,W_{i-1},Z,c_{i-1}).
\]
If necessary, we increase $e_2$ such that $\restr{Y}{e_2}{\alpha}=Y$.
Choose $k_i$ such that $q_{k_i} = e_2$, choose $l_i$ such that $q'_{l_i}=e_1$. Set $W_i \coloneqq  Y$. It is immediate that (2)-(4) hold for $i=1,\dots,n$. For (1), observe that since $f_{\alpha}(W_{i-1}) < f_{\alpha}(Y) < f_{\alpha}(T)$, we deduce from (b) that
\[
\restr{W_{i}}{q_{k_{i-1}}}{\alpha} = \restr{Y}{q_{k_{i-1}}}{\alpha} = W_{i-1}.
\]
Since (1) holds for $i-1$, we get that for $j=1,\dots,i-2$
\[
\restr{W_{i}}{q_{k_{j}}}{\alpha} = \restr{W_{i-1}}{q_{k_{j}}}{\alpha} = W_j.
\]
Thus (1) holds for $i$.

\medskip

\noindent We have constructed $(k_i)_{i=0,\dots,2n}, (l_i)_{i=1,\dots,2n}$ and  $(W_{i})_{i=0,...,2n}$ satisfying (1)-(4) for each $i=0,1,\dots,2n$. We now define $(Z_1,Z_2,Z_3,Z_4)\in \N^4$ by
$$
Z_1 \coloneqq  \sum_{i=0}^{2n} \. q_{k_i}\,, \ \quad Z_2 \coloneqq  \sum_{i=1}^{2n} \. q'_{l_i}\,, \ \quad  
Z_3 \coloneqq  \sum_{i=0}^n \. q_{k_{2i}}\,, \quad \text{and} \quad Z_4 \coloneqq  W_{2n}\..
$$
Observe that we require the sequences $(k_i)_{i=0,\dots,2n}$ and $(l_i)_{i=1,\dots,2n}$ to be increasing sequences of non-consecutive natural numbers. Therefore the above description of $Z_1,Z_2$ and $Z_3$ immediately gives us the $\alpha$-Ostrowski representations of $Z_1$ and $Z_3$ and the $\beta$-Ostrowski representation of $Z_2$. In particular,
\begin{align}\label{eq:proofadm0}
\begin{split}
\Ost_{\alpha}(Z_1) \. = \. & \bigl\{ q_{k_i} \ : \ i=0,\dots, n\bigr\}, \qquad 
\Ost_{\beta}(Z_2) \. = \. \bigl\{ q'_{l_i} \ : \ i=1,\dots,n\bigr\},\\
&\text{and} \quad \Ost_{\alpha}(Z_3) \. = \. \bigl\{ q_{k_i} \ : \ i=0,\dots, n, \ i \hbox{ even}\bigr\}.
\end{split}
\end{align}
It is now left to prove that for all $s,t\in \N$
\[
(s,t) \in S \Longleftrightarrow \Member(Z_1,Z_2,Z_3,Z_4,s,t).
\]
``$\Rightarrow$'': Let $(s,t) \in S$. Let $i\in \{1,\dots,2n\}$ be such that $(s,t) = (c_i,c_{i+1})$. Observe that $i$ is odd.  We show that
\begin{equation}\label{eq:proofadm}
\Adm(q_{k_{i-1}},q_{k_{i}},q_{k_{i+1}},q_{k_{i-1}},q_{l_i},q_{l_{i+1}},Z_1,Z_2,Z_3,Z_4,c_i,c_{i+1})
\end{equation}
holds. By \eqref{eq:proofadm0} and the fact that $i-1$ is even, we have that
\[
q_{k_{i-1}},q_{k_{i}},q_{k_{i+1}}\in \Ost_{\alpha}(Z_1), \ q'_{l_i},q'_{l_{i+1}} \in \Ost_{\beta}(Z_2), \ q_{k_{i-1}} \in \Ost_{\alpha}(Z_3).
\]
Trivially, $q_{k_{i-1}}\leq q_{k_{i-1}} < q_{k_i}$. By (3) $q_{k_{i-1}}< q'_{l_{i}} < q_{k_i}< q'_{l_{i+1}} < q_{k_{i+1}}.$ Now observe that by (1) we have
\begin{align*}
\restr{Z_4}{q_{k_{i-1}}}{\alpha} &= \restr{W_{2n}}{q_{k_{i-1}}}{\alpha} = W_{i-1},\\
\restr{Z_4}{q_{k_{i}}}{\alpha} &=  \restr{W_{2n}}{q_{k_{i}}}{\alpha} = W_{i},\\
\restr{Z_4}{q_{k_{i+1}}}{\alpha} &=  \restr{W_{2n}}{q_{k_{i+1}}}{\alpha} = W_{i+1}.
\end{align*}
Thus by (4)
\[
\Best (q_{k_{i-1}},q'_{l_{i}},\restr{Z_4}{q_{k_{i-1}}}{\alpha},Z_4,c_i) \wedge \Best (q_{k_{i}},q'_{l_{i+1}},\restr{Z_4}{q_{k_{i}}}{\alpha} ,Z_4,c_{i+1}).
\]
Thus \eqref{eq:proofadm} holds.

\medskip

``$\Leftarrow$'': Suppose that $\Member(Z_1,Z_2,Z_3,Z_4,s,t)$ holds. Let $d_1,d_2,d_3,d_4 \in \Ost_{\alpha},e_1,e_2\in \Ost_{\beta}$ be such that
\begin{equation}\label{eq:proofadm1}
\Adm(d_1,d_2,d_3,d_4,e_1,e_2,Z_1,Z_2,Z_3,Z_4,s,t)
\end{equation}
holds. Then $d_1,d_2,d_3$ are consecutive elements of $\Ost_{\alpha}(Z_1)$. Thus there is $i\in \{1,\dots,2n-1\}$ such that
\[
d_1 \coloneqq  q_{k_{i-1}}, \ d_2 \coloneqq  q_{k_{i}}, \ d_3 \coloneqq  q_{k_{i+1}}.
\]
Since $d_4 \in \Ost_{\alpha}(Z_3)$ and $d_1 \leq d_4 < d_2$, it follows that $d_4 = d_1 = q_{k_{i-1}}$ and that $i$ is odd. Since $e_1,e_2 \in \Ost_{\beta}(Z_2)$ and
\[
d_1=q_{k_{i-1}} \leq e_1 < d_2= q_{k_{i}} \leq e_2 \leq d_3 = q_{k_{i+1}},
\]
we get from (3) that $e_1 = q'_{ l_i}$ and $e_2 = q'_{l_{i+1}}$. Thus by \eqref{eq:proofadm1}
\[
\Best (q_{k_{i-1}},q'_{l_{i}},\restr{Z_4}{q_{k_{i-1}}}{\alpha} ,Z_4,s) \wedge \Best (q_{k_{i}},q'_{l_{i+1}},\restr{Z_4}{q_{k_{i}}}{\alpha} ,Z_4,t).
\]
By (4) we get that $s=c_{i}$ and $t=c_{i+1}$. Since $i$ is odd, $(s,t)=(c_{i},c_{i+1}) \in S$.
\end{proof}

\begin{lem}\label{lem:ea} $\Adm$ and $\Member$ are definable by $\exists^{*}\forall^{*}$ $\Q(\alpha,\beta)$-PA formulas.
\end{lem}
\begin{proof}
For $\Adm$ (Definition~\ref{def:Adm_Mem}), we replace each variable $d_{i}$, which earlier represented some convergent $q_{n} \in \Ost_{\alpha}$, by a $6$-tuple
$\cj d_{i} = (u^{-}_{i},v^{-}_{i},u_{i},v_{i},u^{+}_{i},v^{+}_{i})$ such that:
\begin{equation}\label{eq:6-tuple}
(u^{-}_{i},v^{-}_{i},u_{i},v_{i},u^{+}_{i},v^{+}_{i}) = (p_{n-1},q_{n-1},p_{n},q_{n},p_{n+1},q_{n+1}) \;\; \text{for some} \; n.
\end{equation}
We require that $\C_{\forall,\alpha}(u^{-}_{i},v^{-}_{i},u_{i},v_{i},u^{+}_{i},v^{+}_{i})$ holds, in order to guarantee~\eqref{eq:6-tuple}.
Here $v_{i}$ takes the earlier role of $d_{i}$.
Similarly, we replace each $e_{i}$ in $\Adm$ by a $6$-tuple $\cj e_{i}$ and also require that $\C_{\forall,\beta}(\cj e_{i})$ holds.
Here $\C_{\forall,\alpha}$ and $\C_{\forall,\beta}$ are from~\eqref{eq:multi_consec}, with the extra subscript $\alpha$ or $\beta$ indicating which irrational is being considered.
These $\C_{\forall,\alpha}$ and $\C_{\forall,\beta}$ conditions can be combined into a $\for^{2}$-part.
Altogether, the new $\Adm$ has $42$ variables.

\medskip

\nin
Recall that $\Best$ is $\exists^{5} \forall^{4}$-definable (Lemma~\ref{lem:Best}).
The relation $Y = \restr{X}{\d}{\alpha}$ from~\eqref{eq:restr} is $\exists^{2}$-definable.
Here $\Comp$ is from~\eqref{eq:Compatible}.

\medskip

\nin
The relation $\cj d \in \Ost_{\alpha}(X)$, meaning $v$ appears in $\Ost_{\alpha}(X)$, is $\ex^{3}$-definable (see~\eqref{eq:In_ex}).
The same holds for $\cj e \in \Ost_{\beta}(X)$ (just replace $\alpha$ by $\beta$).

\medskip

\nin
The relation
$$
\gathered
\Cons_{\ex}(\cj d_{1},\cj d_{2},X) \; \coloneqq \; v_{1} < v_{2} \wedge \cj d_{1} \in \Ost_{\alpha}(X) \wedge \cj d_{2} \in \Ost_{\alpha}(X) \wedge \\
\ex Y_{1},Y_{2} \;\; Y_{1} = \restr{X}{\cj d_{1}}{\alpha} \; \land \; Y_{2} = \restr{X}{\cj d_{2}}{\alpha} \wedge \hyperref[eq:After]{\After}(u^{-}_{2},v^{-}_{2},u_{i},v_{i},Y_{2}-Y_{1})
\endgathered
$$
means  $v_{1} < v_{2}$  appear consecutively in $\Ost_{\alpha}(X)$.
This is $\ex^{12}$-definable.


\medskip

\nin
It is now easy to see that $\Adm$ is $\exists^{*}\forall^{*}$-definable, and so is $\Member$.
A direct count reveals that $\Adm$ is at most $\ex^{50}\for^{10}$, and $\Member$ is at most $\ex^{100}\for^{10}$.
\end{proof}

\subsection{Proof of Theorem \ref{thm:undtwo}}
  Here we follow an argument given in the proof of Thomas \cite[Th.~16.5]{Thomas}.
  Consider $U = (\States, \Alph, \sigma_1, \delta, q_{1},q_{2})$ a universal $1$-tape Turing machine with $8$ states and $4$ symbols, as given in~\cite{NW}.
  Here $\States = \{q_{1},\dots,q_{8}\}$ are the states, $\Alph = \{\sigma_{1},\dots,\sigma_{4}\}$ are the tape symbols, $\sigma_{1}$ is the blank symbol,
  $q_{1}$ is the start state and $q_{2}$ is the unique halt state.
  Also, $\delta : [8] \times [4] \to [8] \times [4] \times \{\pm 1\}$ is the transition function.
  In other words, we have $\delta(i,j) = (i',j',d)$ if upon state $q_{i}$ and symbol $\sigma_{j}$, the machine changes to state $q_{i'}$, writes symbol $\sigma_{j'}$ and moves left ($d=-1$) or right ($d=1$).
  Given an input $x \in \Alph^{*}$, we will now produce an $\exists^{*} \forall^{*} \exists^{*} \forall^{*}$ $\alpha$-PA sentence $\varphi_{x}$ such that $\varphi_{x}$ holds if and only if $U(x)$ halts.

\medskip


\nin  We will now use sets $A_1,\dots,A_8 \subseteq \N^2$ and $B_1,\dots B_4 \subseteq \N^2$ to code the computation on $U(x)$.
  The $A_i$'s code the current state of the Turing machine.
  That is, for $(s,t) \in \N^{2}$, we have $(s,t) \in A_i$ if and only if at step $s$ of the computation, $U$ is in state $q_{i}$ and its head is over the $t$-th cell of the tape.
  The $B_j$'s code which symbols are written on the tape at a given step of the computation.
  We have $(s,t) \in B_j$ if and only if at step $s$ of the computation, the symbol $\sigma_{j}$ is written on $t$-th cell of the tape.  The computation $U(x)$ then halts if and only if there are $A_1,\dots,A_8 \subseteq \N^2$ and $B_1,\dots B_4 \subseteq \N^2$ such that:
\begin{enumerate}
\item[a)] $A_{i}$'s are pairwise disjoint; $B_{j}$'s are pairwise disjoint.
\item[b)] $(0,0) \in A_1$, i.e., the computation starts in the initial state.
\item[c)] There exists some $(u,v) \in A_2$, i.e., the computation eventually halts.
\item[d)] For each $s \in \N$, there is at most one $t \in \N$ such that $(s,t) \in \cup_{i} A_i$, i.e., at each step of the computation, $U$ can only be in exactly one state.
\item[e)] If $x = x_{0} \dots x_{n} \in \Alph^{*}$, then for every $0 \le t \le n$, we have $x_{t} = \sigma_{j} \iff (0,t) \in B_{j}$, i.e., the first rows of the $B_j$'s code the input string $x$.
\item[f)] Whenever $(s,t) \in B_j$,
\begin{enumerate}
  \item[f1)] if  $(s,t)\notin A_i$ for all $i \in [8]$, then $(s+1,t)\in B_{j}$. That is, if the current head position is not at $t$, then the $t$-th symbol does not change.
  \item[f2)] if  $(s,t)\in A_i$ for some $i \in [8]$ and $\delta(i,j) = (\delta^{1}_{ij},\delta^{2}_{ij},\delta^{3}_{ij}) \in [8] \times [4] \times \{\pm 1\}$, then $(s+1,t)\in B_{\delta^{2}_{ij}}$ and $(s+1,t + \delta^{3}_{ij}) \in A_{\delta^{1}_{ij}}$. That is, if the head position is at $t$, and the state is $i$, then a transition rule is applied.
\end{enumerate}
\end{enumerate}

\nin We use the predicate $\Member$ to code membership $(s,t) \in A_{i}, B_{j}$.
By Theorem~\ref{thm:codepairs}, there should exist  tuples $\bm{X}_{i} = (X_{i1},\dots,X_{i4}),\, \bm{Y}_{j} = (Y_{j1},\dots,Y_{j4}) \in \N^{4}$ that represent $A_{i}$ and $B_{j}$.
In other words, we have $$(s,t) \in A_{i} \iff \Member(\bm{X}_{i},s,t) \;,\quad (s,t) \in B_{j} \iff \Member(\bm{Y}_{j},s,t).$$
For the input condition $e)$, there exist $\bm{Z}_{j} = (Z_{j1},\dots,Z_{j4}) \in \N^{4}$ so that
$$x_{t}=\sigma_{j} \iff \Member(\bm{Z}_{j},0,t) \quad \forall \; 0 \le t \le n.$$
Note that $\bm{Z}_{j}$ can be explicitly constructed from the input $x$ (see~Theorem~\ref{thm:codepairs}'s proof).
Now the sentence $\phi_{x}$ that encodes halting of $U(x)$ is:
\allowdisplaybreaks
\begin{align*}
  \varphi_x \; \coloneqq \;
   & \exists \bm{X}_1,  \dots , \bm{X}_8,\,  \bm{Y}_1,\dots, \bm{Y}_4 \in \N^4, \, u,v \in \N \quad \forall s,t,t' \in \N  \\
& \bigwedge_{i \neq i'} \lnot \big( \Member(\bm{X}_{i},s,t) \wedge \Member(\bm{X}_{i'},s,t) \big)\\
\wedge & \bigwedge_{j \neq j'} \lnot \big( \Member(\bm{Y}_{j},s,t) \wedge \Member(\bm{Y}_{j'},s,t) \big)\\
\wedge & \Member(\bm{X}_{1},0,0) \; \wedge \; \Member(\bm{X}_{2},u,v)\\
\wedge &  \Big[ \big( \bigvee_{i} \Member(\bm{X}_i,s,t) \big) \wedge  \big( \bigvee_{i} \Member(\bm{X}_i,s,t') \big) \; \to \; t = t' \Big]\\
\wedge &  \bigwedge_{j} \big( \Member(\bm{Z}_{j},0,t) \to \Member(\bm{Y}_{j},0,t) \big)\\
\wedge &  \bigwedge_{j} \Big( \; \Member(\bm{Y}_{j},s,t)  \to \big[ \, \bigwedge_{i} \lnot \Member(\bm{X}_i,s,t) \wedge \Member(\bm{Y}_{j},s+1,t) \,\big]\\
\vee \bigvee_{i} \big[ \, & \Member(\bm{X}_i,s,t) \wedge \Member(\bm{Y}_{\delta^{2}_{ij}},s+1,t) \wedge \Member(\bm{X}_{\delta^{1}_{ij}},s+1,t+\delta^{3}_{ij})\, \big] \; \Big).\\
\end{align*}
Since $\Member$ is $\ex^{*}\for^{*}$-definable, the sentence $\phi_{x}$ is $\ex^{*}\for^{*}\ex^{*}\for^{*}$.
Whether $U(x)$ halts or not is undecidable, and so is $\phi_{x}$.
A direct count shows that $\Member$ appears at most $200$ times in $\phi_{x}$.
From the last estimate in the proof of Lemma \ref{lem:ea}, we see that
$\phi_{x}$ is at most a $\ex^{k} \for^{k} \ex^{k} \for^{k}$ sentence, where
$k = 20000$.  This completes the proof. \ $\sq$

\bigskip

\section{Final remarks and open problems} \label{s:fin-rem}

\subsection{}\label{ss:finrem-open}
Comparing Theorem~\ref{t:alg-IP} and Theorem~\ref{th:quad-pspace},
we see a big complexity jump by going from one to three
alternating quantifier blocks, even when $\al$ is quadratic.
It is an interesting open problem to determine the complexity of $\al$-PA sentence when $r=2,3$. Here we make the following conjecture:

\begin{conj}
Let $\al$ be non-quadratic. Then $\al$-\textup{PA} sentences with three alternating blocks of quantifiers are undecidable.
\end{conj}

\noindent Similarly, when $\al$ is quadratic we make the following conjecture:

\begin{conj}
Let $\al$ be quadratic. Then deciding $\al$-\textup{PA} sentences with two alternating blocks of quantifiers and a fixed number of variables and inequalities is \NP-hard.
\end{conj}

\noindent We note that $\ex^{*}\for^{*}$ $\sqrt{5}$-PA sentences can already express non-trivial questions,
such as the following: \ts \emph{Given $a,b \in \Z$, decide whether there is a Fibonacci number $F_{n}$
congruent to $a$ modulo $b$?}  Note that the sequence $\{F_n~\text{mod}~b\}$ is periodic with period $O(b)$,
called the \emph{Pisano period}.  These periods were introduced by Lagrange and heavily studied
in number theory (see e.g.~\cite[$\S$29]{Sil}), but the question above is likely computationally
hard.

\subsection{} \label{ss:finrem-KP}
The main theorem by Khachiyan and Porkolab in~\cite{KP}
is the following general integer optimization result
on convex semialgebraic sets.

\begin{thmC}[\cite{KP}]\label{th:KP}
Consider a first order formula $F(\y)$ over the reals of the form:
\begin{equation*}
\y \in \R^{k}  : Q_{1}\ts\x_{1} \in \R^{n_{1}} \; \dots \; Q_{m}\ts\x_{m} \in \R^{n_{m}} \; P(\y,\x_{1},\dots,\x_{m}),
\end{equation*}
where $P(\y,\x_{1},\dots,\x_{w})$ is a Boolean combination of equalities/inequalities of the form $$g_{i}(\y,\x_{1},\dots,\x_{w}) \; *_{i} \; 0$$ with $*_{i} \in \{>,<,=\}$ and $g_{i} \in \Z[\y,\x_{1},\dots,\x_{w}]$.
Let $k,m,n_{1},\dots,n_{m}$ be fixed, and suppose that the set
$$S_{F} \. \coloneqq \. \{\y \in \rr^{n} : F(\y)=\text{\rm true}\}$$
is convex.  Then we can either decide in polynomial time that $S_{F} \cap \Z^{k} = \varnothing$,
or produce in polynomial time some $\y \in S_{F} \cap \Z^{k}$.
\end{thmC}

\noindent This immediately implies Theorem~\ref{t:alg-IP}.
Here there is no restriction on the number of $g_{i}$'s and their degrees.
The coefficients of $g_{i}$'s are encoded in binary. Note that convexity is crucially important in the theorem.
In Manders and Adleman~\cite{MA}, it is shown that given $a,b,c \in \zz$,
deciding $\ex \y \in \N^{2} \, : \, ay_{1}^{2} + by_{2} + c = 0$ is
$\NP$-complete. Here the semialgebraic set
\begin{equation*}
\bigl\{\y \in \rr^{2} \; : \; 0 \le ay_{1}^{2} + by_{2} + c < 1\bigr\}
\end{equation*}
is not necessarily convex.

\section*{Acknowledgements}
We are grateful to Matthias Aschenbrenner, Sasha Barvinok, Tristram Bogart, Art\"{e}m Chernikov,
Jes\'{u}s De Loera, Fritz Eisenbrand, John Goodrick, Robert Hildebrand, Ravi Kannan,
Matthias K\"oppe, and Kevin Woods for interesting conversations and helpful remarks. We thank Eion Blanchard, Madie Farris and Ran Ji for closely reading an earlier draft of this manuscript.
This work was finished while the last two authors were in residence
of the MSRI long term Combinatorics program in the Fall of 2017;
we thank MSRI for the hospitality.  The first author was partially supported by NSF Grant DMS-1654725.
The second author
was partially supported by the UCLA Dissertation Year Fellowship.
The third author was partially supported by the~NSF.


\vskip.82cm

\begin{thebibliography}{21132}\label{refpage}

\bibitem[Bar]{B2}
A.~Barvinok,
The complexity of generating functions for integer points in polyhedra and beyond,
in \emph{Proc.~ICM}, Vol.~3, EMS, Z\"urich, 2006, 763--787.


\bibitem[Ber]{Ber}
L.~Berman, The complexity of logical theories,
\emph{Theoret.\ Comput.\ Sci.}~\textbf{11} (1980), 71--77.

\bibitem[Coo]{C}
D.~C.~Cooper,
Theorem proving in arithmetic without multiplication,
in \emph{Machine Intelligence} (B.~Meltzer and D.~Michie, eds.),
Edinburgh Univ.\ Press, 1972, 91--99.



\bibitem[FR]{FR}
M.~J.~Fischer and M.~O.~Rabin,
Super-Exponential Complexity of Presburger Arithmetic, in
\emph{Proc.\ SIAM-AMS Symposium in Applied Mathematics}, AMS,
Providence, RI, 1974, 27--41.

\bibitem[F\"{u}r]{Fur}
M.~F\"{u}rer,
The complexity of Presburger arithmetic with bounded quantifier alternation depth,
\emph{Theoret.\ Comput.\ Sci.}~\textbf{18} (1982), 105--111.


\bibitem[Gr\"{a}]{Gra}
E.~Gr\"{a}del,
\emph{The complexity of subclasses of logical theories}, Ph.D.~thesis,
Universit\"{a}t Basel, 1987.



\bibitem[H1]{H2}
P.~Hieronymi,
When is scalar multiplication decidable, 
\emph{Ann.\ Pure Appl.\ Logic}~{\bf 170} (2019), 1162--1175. 

\bibitem[H2]{H}
P.~Hieronymi,
Expansions of the ordered additive group of real numbers by two discrete subgroups,
\emph{J.~Symb.\ Log.}~{\bf 81} (2016), 1007--1027.



\bibitem[HTe]{HT}
P.~Hieronymi and A.~Terry~Jr.,
Ostrowski Numeration Systems, Addition, and Finite Automata,
\emph{Notre Dame J.\ Form.\ Log.}~\textbf{59} (2018), 215--232.



\bibitem[HTy]{HT-Proj}
P.~Hieronymi and M.~Tychonievich,
Interpreting the projective hierarchy in expansions of the real line,
\emph{Proc.\ Amer.\ Math.\ Soc.}~{\bf 142} (2014), 3259--3267.


\bibitem[HUM]{HUM}
J.~E.~Hopcroft, J.~Ullman and R.~Motwani,
\emph{Introduction to automata theory, languages, and computation} (3rd~ed.),
Addison-Wesley, 2006.


\bibitem[Kan]{K1}
R.~Kannan,
Test sets for integer programs, $\forall\ts\exists$ sentences,
in \emph{Polyhedral Combinatorics}, AMS, Providence, RI, 1990,
39--47.



\bibitem[KP]{KP}
L.~Khachiyan and L.~Porkolab,
Integer optimization on convex semialgebraic sets,
\emph{Discrete Comput.\ Geom.}~{\bf 23} (2000), 207--224.


\bibitem[KN]{automata}
B.~Khoussainov and A.~Nerode,
\emph{Automata Theory and its Applications}, 
Birkh{\"a}user, Boston, MA, 2001, 430~pp.


\bibitem[Len]{L}
H.~Lenstra,
Integer programming with a fixed number of variables,
\emph{Math.\ Oper.\ Res.}~\textbf{8} (1983), 538--548.

\bibitem[MA]{MA}
K.~Manders and L.~Adleman,
$\NP$-complete decision problems for binary quadratics,
\emph{J.~Comput.\ System Sci.}~{\bf 16} (1978), 168--184.


\bibitem[Mey]{M}
A.~Meyer,
Weak monadic second order theory of succesor is not elementary-recursive,
in \emph{Lecture Notes in Math.}~\textbf{453}, Springer, Berlin, 1975, 132--154.



\bibitem[Mil]{ivp}
C.~Miller,
Expansions of dense linear orders with the intermediate value property,
{\em J.~Symbolic Logic}~\textbf{66} (2001), 1783--1790.


\bibitem[NW]{NW}
  T.~Neary and D.~Woods,
  Small fast universal Turing machines,
\emph{Theor.\ Comput.\ Sci.}~\textbf{362} (2006), 171--195.

\bibitem[NP]{NP}
D.~Nguyen and I.~Pak,
Short Presburger arithmetic is hard, to appear in {\em SIAM Jour.\ Comp.}; 
extended abstract in \emph{58-th Proc.\ FOCS}, Los Alamitos, CA, 2017, 37--48; 
\texttt{arXiv:1708.08179}.


\bibitem[Opp]{Oppen}
D.~C.~Oppen,
A $2^{2^{2^{pn}}}$ upper bound on the complexity of Presburger arithmetic,
\emph{J.\ Comput.\ System Sci.}~\textbf{16} (1978), 323--332.


\bibitem[Ost]{Ost}
A.~Ostrowski,
 Bemerkungen zur {T}heorie der {D}iophantischen {A}pproximationen (in German),
 {\em Abh.\ Math.\ Sem.\ Univ.\ Hamburg}~\textbf{1} (1922), 77--98.


\bibitem[Pre]{Pres}
M.~Presburger, \"{U}ber die Vollst\"{a}ndigkeit eines gewissen Systems der Arithmetik ganzer Zahlen,
in welchem die Addition als einzige Operation hervortritt (in German),
in \emph{Comptes Rendus du I congr\`{e}s de Math\'{e}maticiens des Pays Slaves},
Warszawa, 1929, 92--101.




\bibitem[RL]{RL}
C.~R.~Reddy and D.~W.~Loveland,
Presburger arithmetic with bounded quantifier alternation,
in \emph{Proc.\ 10th STOC}, ACM, 1978, 320-325.

\bibitem[Rei]{GTW}
K.~Reinhardt, The complexity of translating logic to finite automata, 
in \emph{Automata, Logics, and Infinite Games. A Guide
to Current Research}, 
    Springer, Berlin, 2002, 231--238.


\bibitem[RS]{RS}
  A.~M.~Rockett and P.~Sz\"{u}sz,
  \emph{Continued fractions}, World Sci., River Edge, NJ, 1992.

\bibitem[Sca]{Sca}
B.~Scarpellini,
Complexity of subcases of Presburger arithmetic,
\emph{Trans.\ AMS}~\textbf{284} (1984), 203--218.

\bibitem[Schr]{Schrijver}
A.~Schrijver,
\emph{Theory of linear and integer programming},
John Wiley, Chichester, 1986.

\bibitem[Sch\"{o}]{Sch}
U.~Sch\"{o}ning,
Complexity of Presburger arithmetic with fixed quantifier dimension,
\emph{Theory Comput.\ Syst.}~\textbf{30} (1997), 423--428.

\bibitem[Sil]{Sil}
J.~H.~Silverman,
\emph{A Friendly Introduction to Number Theory}, Pearson, 2011.


\bibitem[Sko]{skolem}
T.~Skolem,
{\"U}ber einige {S}atzfunktionen in der {A}rithmetik (in German), in
{\em Skr.\ Norske Vidensk.\ Akad., Oslo,
Math.-naturwiss.\ Kl.}~\textbf{7} (1931), 1--28.


\bibitem[Sto]{S}
 L.~Stockmeyer,
 \emph{The Complexity of Decision Problems in Automata Theory and Logic},
 Ph.D.\ thesis, MIT, 1974, 224 pp.



\bibitem[Tho]{Thomas}
W.~Thomas, Finite automata and the analysis of infinite transition systems,
in {\em Modern Applications of Automata Theory}, World Sci.,
Hackensack, NJ, 2012, 495--527.


\bibitem[W1]{Wei}
V.~D.~Weispfenning,
Complexity and uniformity of elimination in Presburger arithmetic,
in \emph{Proc.\ 1997 ISSAC}, ACM, New York, 1997, 48--53.

\bibitem[W2]{weis}
V.~D.~Weispfenning,
Mixed real--integer linear quantifier elimination, in
\emph{Proc.\ 1999 ISSAC}, ACM, New York, 1999, 129--136.


\bibitem[Woo]{Woods}
K.~Woods,
\emph{Rational Generating Functions and Lattice Point Sets},
Ph.D.~thesis, University of Michigan, 2004, 112~pp.


\end{thebibliography}
\end{document}